\documentclass[12pt,reqno]{article}

\usepackage[usenames]{color}
\usepackage{amssymb}
\usepackage{graphicx}
\usepackage{amscd}

\usepackage[colorlinks=true,
linkcolor=webgreen,
filecolor=webbrown,
citecolor=webgreen]{hyperref}

\definecolor{webgreen}{rgb}{0,.5,0}
\definecolor{webbrown}{rgb}{.6,0,0}

\usepackage{color}
\usepackage{fullpage}
\usepackage{float}

\usepackage{psfig}
\usepackage{graphics,amsmath,amssymb}
\usepackage{amsthm}
\usepackage{amsfonts}
\usepackage{latexsym}
\usepackage{epsf}

\usepackage{algorithmicx}
\usepackage{algpseudocode}
\usepackage{tikz}
\usetikzlibrary{automata,positioning}

\newcommand{\seqnum}[1]{\href{http://oeis.org/#1}{\underline{#1}}}

\begin{document}

%\begin{center}
%\epsfxsize=4in
%\leavevmode\epsffile{logo129.eps}
%\end{center}

\theoremstyle{plain}
\newtheorem{theorem}{Theorem}
\newtheorem{corollary}[theorem]{Corollary}
\newtheorem{lemma}[theorem]{Lemma}
\newtheorem{proposition}[theorem]{Proposition}

\theoremstyle{definition}
\newtheorem{definition}[theorem]{Definition}
\newtheorem{example}[theorem]{Example}
\newtheorem{conjecture}[theorem]{Conjecture}

\theoremstyle{remark}
\newtheorem{remark}[theorem]{Remark}

\begin{center}
\vskip 1cm{\LARGE\bf 
Abelian Complexity Function of the \\
\vskip .11in
Tribonacci Word
}
\vskip 1cm
\large
Ond\v{r}ej~Turek \\
Nuclear Physics Institute \\
Academy of Sciences of the Czech Republic \\
250 68 \v{R}e\v{z} \\
Czech Republic \\
and \\
Bogolyubov Laboratory of Theoretical Physics \\
Joint Institute for Nuclear Research \\
141980 Dubna \\
Russia \\
\href{mailto:o.turek@ujf.cas.cz}{\tt o.turek@ujf.cas.cz} \\
\end{center}

\vskip .2 in

\def\N{\mathbb N}
\def\A{\mathcal A}
\def\Codec{\mathrm{Dec}}
\def\Z{\mathcal Z}
\def\D{\mathcal D}
\def\P{\mathcal P}
\def\u{\mathbf u}
\def\t{\mathbf t}
\def\q{\mathbf q}
\def\Prel{\mathcal{P}^{\mathrm{rel}}}
\def\AC{\rho^{\mathrm{ab}}}
\def\Vect{\mathrm{Vect}}

\begin{abstract}
According to a result of Richomme, Saari and Zamboni, the 
abelian complexity of the Tribonacci word satisfies $\rho^{\mathrm{ab}}(n)\in\{3,4,5,6,7\}$ for each $n\in\mathbb{N}$. In this paper we derive an automaton that evaluates the function $\rho^{\mathrm{ab}}(n)$ explicitly. The automaton takes the Tribonacci representation of $n$ as its input; therefore, $(\rho^{\mathrm{ab}}(n))_{n\in\N}$ is an automatic sequence in a generalized sense. Since our evaluation of $\rho^{\mathrm{ab}}(n)$ uses $\mathcal{O}(\log n)$ operations, it is fast even for large values of $n$. Our result also leads to a solution of an open problem proposed by Richomme et al. concerning the characterization of 
those $n$ for which $\rho^{\mathrm{ab}}(n)=c$ with $c$ belonging to $\{4,5,6,7\}$. In addition, we apply the same approach on the $4$-bonacci word. In this way we find a description of the abelian complexity of the $4$-bonacci word, too.
\end{abstract}

\section{Introduction}

The {\it abelian complexity\/} of a word $\u$ is a function $\N\to\N$ that counts the number of pairwise non-abelian-equivalent factors of $\u$ of length $n$.
The term was introduced by Richomme, Saari and Zamboni \cite{RSZ} in 2009, and since then it has been extensively studied~\cite{ABJS,BBT,CRSZ,CR,MR,Tu2,Tu13}. 
In one of the first papers on the subject, Richomme, Saari and Zamboni~\cite{RSZ2} examined the Tribonacci word $\t$ \seqnum{A080843}, which is the fixed point of the substitution $0\mapsto01$, $1\mapsto02$, $2\mapsto0$, and they showed that $\AC_\t(n)\in\{3,4,5,6,7\}$ for all $n$. They also characterized those
$n$ for which $\AC_\t(n)=3$, and proposed the following open problem: 
{\it for each $c\in\{4,5,6,7\}$,
characterize those $n$ for which $\AC_\t(n)=c$}.

Explicit characterization of $\AC_\u(n)$ of a given infinite word $\u$ is generally a difficult task, particularly in case of words defined over alphabets consisting of more than two letters.
For example, despite the fact that a recurrent word over a ternary alphabet with constant abelian complexity equal to $3$ for all $n\in\N$ has been already constructed~\cite{RSZ}, there seems to be no other nontrivial example to date of a recurrent $m$-ary word with $m\geq3$ whose abelian complexity function has been precisely determined. In particular, the problem of precise characterization of the abelian complexity $\AC_\t(n)$ of the Tribonacci word $\t$, which is a ternary word, has remained open since 2009.

Recently Mousavi and Shallit~\cite{MS} showed that many properties of the Tribonacci word, such as the aperiodicity, powers, palindromes etc., could be examined purely mechanically with the help of finite automata. Although
in principle their method could also be used 
for the study of the characteristics of the abelian complexity function of the Tribonacci word, it turns out to be not computationally feasible. In this paper we propose a related method that is particularly
designed for dealing with abelian properties (primarily with abelian complexity and the balance properties). Our approach is less general than the one of Mousavi and Shallit, but it is efficient enough to explicitly obtain a finite automaton that computes
the function $\AC_\t(n)$. The automaton in question features a very small set of states, consisting of less than $70$ elements. Consequently, it can be easily implemented on any computer (a powerful machine is not needed), and the calculation can be even performed by hand. The automaton takes the Tribonacci representation of $n$ as its input, which means that the sequence $(\AC_t(n))_{n\in\N}$ \seqnum{A216190} is $T$-automatic (or ``Tribonacci-automatic'') in the sense of Shallit~\cite{AS,Sh88}.

Our approach relies on the technique of abelian co-decomposition~\cite{Tu13}, which was originally developed as a tool for proving that the abelian complexity of a recurrent word attains a certain value infinitely often. As a result, our construction of the automaton can be generalized to certain other words as well.

The paper is organized as follows. In Section~\ref{Sect.Preliminaries}
we provide the necessary notation related to abelian complexity, the
Tribonacci word and finite automata. Section~\ref{Sect.Codecomposition}
summarizes basic facts about abelian co-decomposition.
Section~\ref{Sect.Z} contains the core of the paper: we show that there
exists a finite number of sets $\Z_1,\ldots,\Z_M$ such that each
$n\in\N$ can be associated with a certain $\Z_q$ via the Tribonacci
representation of $n$.  Since each of those sets is related to a
certain value of the abelian complexity, the indices $1,\ldots,M$ can
be regarded as states of a finite automaton that evaluates
$\AC_\t(n)$.  Section~\ref{Sect.Range} is devoted to the first and easy
application of the findings from Section~\ref{Sect.Z}: we demonstrate
that the abelian complexity of the Tribonacci word takes values in
$\{3,4,5,6,7\}$.  Although this fact is already known~\cite{RSZ2}, it
illustrates the applicability of our approach.  The main result of the
paper is presented in Section~\ref{Sect.Eval.AC}. We derive a formula
for evaluating the abelian complexity of the Tribonacci word on the
basis of results of Section~\ref{Sect.Z}. In particular, we show that
the abelian complexity can be calculated by a finite automaton with
$278$ states. This result is further improved in
Section~\ref{Sect.Simplified}, where the size of the automaton is
reduced from $278$ to $68$ states.  It is easy to transform this
automaton into an automaton that decides, for any $n\in\N$, whether
$\AC_\t(n)$ attains a given value $c\in\{3,4,5,6,7\}$, which provides a
solution of the problem of Richomme et al. In Section~\ref{Sect.mbon},
we demonstrate that the method allows one to examine the abelian
complexity function of other $m$-bonacci words. In particular, we
present results on the $4$-bonacci word; they show that the
abelian complexity of $m$-bonacci words gains new properties when $m$
exceeds $3$. The paper is concluded by Section~\ref{Sect.Conclusions},
in which we discuss other applications and generalizations of the
method.

\section{Preliminaries}\label{Sect.Preliminaries}

Let us consider a set $\A=\{0,1,2,\ldots,m-1\}$ (\emph{alphabet}) consisting of $m$ symbols (\emph{letters}) $0,1,\ldots,m-1$.
Concatenations of letters from $\A$ are called \emph{words}. Let $\A^*$ denote the free monoid of all finite words over $\A$ including the empty word $\varepsilon$. The \emph{length} of a word $w=w_0w_1w_2\cdots w_{n-1}\in\A^*$ is the number of its letters, $|w|=n$. The symbol $|w|_\ell$ for $\ell\in\A$ and $w\in\A^*$ stands for the number of occurrences of the letter $\ell$ in the word $w$.

The set of all infinite words over $\A$ is denoted by $\A^\N$. We say that an infinite word $\u$ is \emph{recurrent} if every factor of $\u$ occurs infinitely many times in $\u$.

A finite word $w$ is a \emph{factor} of a (finite or infinite) word $\u$ if there exists a finite word $x$ and a (finite or infinite, respectively) word $y$ such that $\u=xwy$. If $x=\varepsilon$, the factor $w$ is called
a \emph{prefix} of $\u$.

For any word $w\in\A^*$ and $k\in\N$ we write $w^k=\overbrace{ww\cdots w}^{k \text{ times}}$. Similarly, we set $w^0=\varepsilon$. If a word $v\in\A^\N$ has the prefix $w^k$ for $k\in\N$, then the symbol $w^{-k}v$ stands for the word satisfying $w^kw^{-k}v=v$.

The \emph{Parikh vector} of $w$ is the $m$-tuple $\Psi(w)=(|w|_0,|w|_1,\ldots,|w|_{m-1})$; note that $|w|_0+|w|_1+\cdots+|w|_{m-1}=|w|$.
For any given infinite word $\u$, let $\mathcal{P}_\u(n)$ denote the set of all Parikh vectors corresponding to factors of $\u$ having the length $n$, i.e.,
$$
\mathcal{P}_\u(n)=\left\{\Psi(w)\,\left|\,\text{$w$ is a factor of $\u$}, |w|=n\right.\right\}.
$$
The \emph{abelian complexity} of a word $\u$ is the function $\AC_\u:\N\to\N$ defined as
\begin{equation}\label{AC}
\AC_\u(n)=\#\P_\u(n),
\end{equation}
where $\#$ denotes the cardinality.

It is useful to introduce the \emph{relative Parikh vector}~\cite{Tu13}, which is defined for any factor $w$ of $\u$ of length $n$ as
$$
{\Psi}_\u^\mathrm{rel}(w)=\Psi(w)-\Psi(\u_{[n]})\,,
$$
where $\u_{[n]}$ is the prefix of $\u$ of length $n$.
Since the subtrahend $\Psi(\u_{[n]})$ depends only on the length of $w$ and not on $w$ itself, the set of relative Parikh vectors corresponding to the length $n$,
$$
\Prel_\u(n):=\left\{\left.{\Psi}_\u^\mathrm{rel}(w)\;\right|\;\text{$w$ is a factor of $\u$}, |w|=n\right\},
$$
has the same cardinality as $\P_\u(n)$. Hence we obtain, with regard to~\eqref{AC},
\begin{equation}\label{ACrel}
\AC_\u(n)=\#\Prel_\u(n)\,.
\end{equation}

An infinite word $\u$ is said to be \emph{$b$-balanced} for a certain $b\in\N$ if for every $\ell\in\A$ and for every pair of factors $v$, $w$ of $\u$ such that $|v|=|w|$, the inequality
$\left||v|_\ell-|w|_\ell\right|\leq b$
holds.
If $\u$ is a $b$-balanced word, the components of relative Parikh vectors are bounded by $b$~\cite{Tu13}. Therefore, the set of all relative Parikh vectors $\bigcup_{n\in\N}\Prel_\u(n)$ is finite for any $b$-balanced word $\u$.

This paper is primarily concerned with the \emph{Tribonacci word} $\t$, which is defined over the alphabet $\A=\{0,1,2\}$ as the fixed point of the substitution
\begin{equation}\label{Trib.subst}
\begin{array}{lccl}
\varphi_\t: & 0 & \mapsto & 01 \\
& 1 & \mapsto & 02 \\
& 2 & \mapsto & 0
\end{array}
\end{equation}
i.e.,
$$
\t=\lim_{k\to\infty}\varphi_\t^k(0)=01020100102010102010010201020100102010102010\cdots\,.
$$
It is easy to check that $\varphi_\t^{j}(0)=\varphi_\t^{j-1}(0)\varphi_\t^{j-2}(0)\varphi_\t^{j-3}(0)$ for every $j\geq3$. Hence, the lengths of factors $\varphi_\t^{j}(0)$ satisfy the recurrence relation $|\varphi_\t^{j}(0)|=|\varphi_\t^{j-1}(0)|+|\varphi_\t^{j-2}(0)|+|\varphi_\t^{j-3}(0)|$. Comparing this relation with the Tribonacci recurrence relation $T_j=T_{j-1}+T_{j-2}+T_{j-3}$, we conclude that $|\varphi_\t^j(0)|=T_{j+3}$ for every $j\in\N\cup\{0\}$, where $\left(T_j\right)_{j\geq0}=(0,0,1,1,2,4,7,\ldots)$ is the sequence of Tribonacci 
numbers \seqnum{A000073}.
Any $n\in\N$ can be written as a sum of Tribonacci numbers with binary coefficients,
\begin{equation}\label{TribExp}
n=\sum_{j=0}^{k} d_j T_{j+3} \qquad \text{for}\quad d_j\in\{0,1\},\; k\in\N\cup\{0\}\,.
\end{equation}
If coefficients $d_j\in\{0,1\}$ are obtained by the greedy algorithm, they form the \emph{normal $T$-representation} (also called the \emph{Tribonacci representation}) of $n$, which we denote by the symbol $\langle n \rangle_T$:
\begin{equation}\label{Trep}
\langle n \rangle_T=d_k d_{k-1}\cdots d_1d_0\,.
\end{equation}
For $n=0$, we have $\langle 0 \rangle_T=\varepsilon$.
Table~\ref{T expansions} shows normal $T$-representations of several small integers.
\begin{table}
\begin{center}
\begin{tabular}{|c|r||c|r||c|r||c|r||c|r|}
\hline
$n$ & $\langle n \rangle_T$ & $n$ & $\langle n \rangle_T$ & $n$ & $\langle n \rangle_T$ & $n$ & $\langle n \rangle_T$ & $n$ & $\langle n \rangle_T$ \\
\hline
$1$ & $1$ & $4$ & $100$ & $7$ & $1000$ & $10$ & $1011$ & $13$ & $10000$ \\
$2$ & $10$ & $5$ & $101$ & $8$ & $1001$ & $11$ & $1100$ & $14$ & $10001$ \\
$3$ & $11$ & $6$ & $110$ & $9$ & $1010$ & $12$ & $1101$ & $15$ & $10010$ \\
\hline
\end{tabular}
\end{center}
\caption{Normal $T$-representations of the numbers $1,\ldots,15$.}
\label{T expansions}
\end{table}
The constant $k$ in expansion~\eqref{Trep} does not need to be chosen minimal, i.e., a normal $T$-representation can start with a block of zeros. For example,
the representations $\langle n \rangle_T=011$ and $\langle n \rangle_T=00011$ are both equivalent to $\langle n \rangle_T=11$ and correspond to $n=3$.

The substitution~\eqref{Trib.subst} is a special case of a \emph{simple Parry substitution}, defined over the alphabet $\A=\{0,1,\ldots,m-1\}$ in the way
\begin{equation}\label{simpleParry}
\begin{array}{lccl}
\varphi: & 0 & \mapsto & 0^{\alpha_0}1 \\
& 1 & \mapsto & 0^{\alpha_1}2 \\
&  & \vdots & \\
& m-2 & \mapsto & 0^{\alpha_{m-2}}(m-1) \\
& m-1 & \mapsto & 0^{\alpha_{m-1}}
\end{array}
\end{equation}
with $\alpha_i\in\N\cup\{0\}$ satisfying the conditions $\alpha_0\geq1$ and $\alpha_\ell\leq \alpha_0$ for all $\ell\in\A$ \cite{Fa,Pa}. We call the fixed point of~\eqref{simpleParry} a \emph{simple Parry word}; in this sense the Tribonacci word is an example of a simple Parry word. Simple Parry words appear in nonstandard numeration systems. Without going into details, let us mention here that the order of letters in the fixed point of~\eqref{simpleParry} corresponds to the order of lengths of gaps between so-called $\beta$-integers for $\beta>1$ being a zero of the polynomial $\alpha_{m-1}x^{m-1}+\alpha_{m-2}x^{m-2}+\cdots+\alpha_1 x+\alpha_0$~\cite{Th89}. Since all simple Parry substitutions have common structure, the combinatorial properties of their fixed points can be often examined in a similar way. In particular, the approach we are going to apply in this paper is based on a method that can handle fixed points of all substitutions of type~\eqref{simpleParry}. Therefore, it is convenient here to formulate the representation~\eqref{Trep} more generally. Consider a simple Parry substitution $\varphi$, and
set $U_j=|\varphi^j(0)|$ for every $j\in\N\cup\{0\}$. 
Any $n\in\N$ can be represented as a sum
$$
n=\sum_{j=0}^{k} d_j U_j
$$
with integer coefficients $d_j$.
If coefficients $d_j$ are obtained by the greedy algorithm, the sequence $d_kd_{k-1}\cdots d_1d_0$ is called the \emph{normal $U$-representation} of $n$ \cite{Lo} and denoted
\begin{equation}\label{Urep}
\langle n \rangle_U=d_kd_{k-1}\cdots d_1d_0\,.
\end{equation}
It can be shown that the coefficients in~\eqref{Urep} satisfy $d_j\in\{0,1,\ldots,\alpha_0\}$ for all $j=0,1,\ldots,k$.

A \emph{deterministic finite automaton with output} (DFAO)~\cite{AS} is an extension of the deterministic finite automaton (DFA) model. A DFAO is defined as a $6$-tuple $(Q,\Sigma,\delta,q_0,\Delta,\tau)$, where $Q$ is a finite set of states, $\Sigma$ is the finite input alphabet, $\delta:Q\times\Sigma\to Q$ is the transition function, $q_0$ is the initial state, $\Delta$ is the output alphabet, and $\tau:Q\to\Delta$ is the output function. If we extend the domain of $\delta$ to $Q\times\Sigma^*$ by defining $\delta(q,\epsilon)=q$ for all $q\in Q$, and $\delta(q,xa)=\delta(\delta(q,x),a)$ for all $q\in Q$, $x\in\Sigma^*$ and $a\in\Sigma$, a DFAO defines a function $f:\Sigma^*\to\Delta$, given as
$$
f(w)=\tau(\delta(q_0,w)) \qquad \text{for $w\in\Sigma^*$}.
$$
Let $[n]_k$ denote the representation of $n\in\N$ in base $k$ for a certain integer $k\geq2$. A sequence $(a_n)_{n\in\N}$ over a finite alphabet $\Delta$ is called \emph{$k$-automatic} if there exists a DFAO with $\Sigma=\{0,1,\ldots,k-1\}$ such that $a_n=\tau(\delta(q_0,[n]_k))$ for all $n\in\N$.

Shallit~\cite{Sh88} introduced the concept of generalized automatic sequences, which are generated by automata using nonstandard representations instead of ordinary base-$k$ representations.
In particular, we say that a sequence $(a_n)_{n\in\N}$ with values in a finite alphabet $\Delta$ is \emph{$U$-automatic} if there exists a DFAO $(Q,\Sigma,\delta,q_0,\Delta,\tau)$ such that
$$
a_n=\tau(\delta(q_0,\langle n\rangle_U)) \qquad \text{for all $n\in\N$},
$$
where $\langle n\rangle_U$ is the normal $U$-representation defined above.

\section{Abelian co-decomposition}\label{Sect.Codecomposition}

Abelian co-decomposition, which we briefly summarize in this section, has been developed as a tool for calculating $\AC_\u(n)$ of recurrent words~\cite{Tu13}. The main idea is roughly the following: the method uses the normal $U$-representation $\langle n \rangle_U$ to associate any $n\in\N$ with a certain set $\Z_\u(n)$ of pairs of factors. At the same time the structure of the set $\Z_\u(n)$ is designed in a way that allows to find quickly the set of relative Parikh vectors $\Prel_\u(n)$, and, consequently, to obtain the value $\AC_\u(n)$ by formula~\eqref{ACrel}.

Let $v,w$ be any factors of $\u$ such that $\Psi(v)=\Psi(w)$ (in particular, $|v|=|w|$). We factorize them as follows:
\begin{equation}\label{rozklad}
\begin{array}{ccc}
v&=&z_0\:z_1\:z_2\:\cdots\:z_h \\
w&=&\tilde{z}_0\:\tilde{z}_1\:\tilde{z}_2\:\cdots\:\tilde{z}_h
\end{array}
\end{equation}
where $z_0,z_1,\ldots,z_h$ and $\tilde{z}_0,\tilde{z}_1,\ldots,\tilde{z}_h$ are non-empty words satisfying $\Psi(\tilde{z}_j)=\Psi(z_j)$ for all $j\in\{0,1,\ldots,h\}$. The set of pairs
\begin{equation}\label{codec}
\left\{\begin{pmatrix}z_0\\ \tilde{z}_0\end{pmatrix},\begin{pmatrix}z_1\\ \tilde{z}_1\end{pmatrix},\begin{pmatrix}z_2\\ \tilde{z}_2\end{pmatrix},\cdots,\begin{pmatrix}z_h\\ \tilde{z}_h\end{pmatrix}\right\}
\end{equation}
is called the \emph{abelian co-decomposition} of the pair $\bigl(\begin{smallmatrix}v\\ w\end{smallmatrix}\bigr)$.

An abelian co-decomposition~\eqref{codec} exists for any $v,w$ such that $\Psi(v)=\Psi(w)$, because one can take, e.g., $\left\{\bigl(\begin{smallmatrix}v\\ w\end{smallmatrix}\bigr)\right\}$. The decomposition~\eqref{rozklad} is in general not unique (see Example~\ref{Ex.Codec} below), but it can be made unique by an additional requirement. Here we will adopt, throughout the whole paper, the following convention: \emph{The number $h$ in equation~\eqref{rozklad} is chosen to be maximal.} The abelian co-decomposition satisfying this requirement will be denoted $\Codec\bigl(\begin{smallmatrix}v\\ w\end{smallmatrix}\bigr)$.

\begin{example}\label{Ex.Codec}
Let $v=0102$, $w=1020$. There exist two possible decompositions~\eqref{rozklad}:
$$
\begin{array}{ccc}
v &=& \overbrace{0102}^{z_0} \\
w &=& \underbrace{1020}_{\tilde{z}_0}
\end{array}
\qquad\text{or}\qquad
\begin{array}{ccc}
v &=& \overbrace{01}^{z_0}\,\overbrace{02}^{z_1} \\
w &=& \underbrace{10}_{\tilde{z}_0}\,\underbrace{20}_{\tilde{z}_1}
\end{array}
$$
Hence the abelian co-decomposition of the pair $\bigl(\begin{smallmatrix}0102\\1020\end{smallmatrix}\bigr)$, obeying our convention of maximality of number of elements, is
$$
\Codec\begin{pmatrix}v\\ w\end{pmatrix}=\left\{\begin{pmatrix}01\\10\end{pmatrix},\begin{pmatrix}02\\20\end{pmatrix}\right\}.
$$
\end{example}

For any pair $\bigl(\begin{smallmatrix}v\\ w\end{smallmatrix}\bigr)$ of factors of $\u$ such that $|v|=|w|$, we introduce the following set of vectors:
\begin{equation}\label{P}
\Vect\begin{pmatrix}v\\w\end{pmatrix}:=\left\{\Psi(s)-\Psi(r)\;\left|\;\text{$r$ is a prefix of $v$, $s$ is a prefix of $w$, $|s|=|r|$}\right.\right\}\,.
\end{equation}
\begin{example}\label{Ex.P}
\begin{align*}
\Vect\begin{pmatrix}0102\\1020\end{pmatrix}=&\bigl\{\,\Psi(1)-\Psi(0),\Psi(10)-\Psi(01),\Psi(102)-\Psi(010),\Psi(1020)-\Psi(0102)\,\bigr\}\\
=&\{(-1,1,0),(0,0,0),(-1,0,1)\}.
\end{align*}
\end{example}

Let $\u$ be the fixed point of~\eqref{simpleParry}. For every $n\in\N$, we define the set~\cite[Def.~3.7 and Prop.~4.8]{Tu13}
\begin{equation}\label{Z(n)}
\Z_\u(n):=\Codec\begin{pmatrix}\varphi^{K+R}(0)\\ \u_{[n]}^{-1}\varphi^{K+R}(0)\u_{[n]}\end{pmatrix}\,,
\end{equation}
where $K$ is any integer such that $n\leq U_K$ and $R$ is a constant that we choose using the formula
\begin{equation}\label{R}
R=m-1+\min\{j\,|\,(\forall\ell\in\A)(\varphi^j(\ell) \text{ has the prefix } 0)\}\,.
\end{equation}
Note that the factors $\varphi^{K+R}(0)$ and $\u_{[n]}^{-1}\varphi^{K+R}(0)\u_{[n]}$ are obviously abelian equivalent, thus we are allowed to consider their abelian co-decomposition. In Proposition~\ref{Z(n)correct} we will see that the right-hand side of~\eqref{Z(n)} is independent of the choice of $K$, i.e., the symbol $\Z_\u(n)$ is well-defined.

\begin{proposition}\label{Z(n)correct}
Let $R$ be given by equation~\eqref{R}. For any $n\in\N$ and for any integer $K$ such that $n\leq U_K$ we have
\begin{equation}\label{Codex_explicit}
\Codec\begin{pmatrix}\varphi^{K+R}(0)\\ \u_{[n]}^{-1}\varphi^{K+R}(0)\u_{[n]}\end{pmatrix}=\bigcup_{\ell\in\A}\Codec\begin{pmatrix}\varphi^{K_0+R-m+1}(\ell)\\ \u_{[n]}^{-1}\varphi^{K_0+R-m+1}(\ell)\u_{[n]}\end{pmatrix}\,,
\end{equation}
where $K_0=\min\{K'\in\N\cup\{0\}\;|\;n\leq U_{K'}\}$.
In particular, the right-hand side of equation~\eqref{Z(n)} is independent of the choice of $K$.
\end{proposition}

\begin{proof}
Let us take an arbitrary $K$ such that $n\leq U_K$. Since we have $R\geq m-1$ by~\eqref{R}, we can write
$$
\varphi^{K+R}(0)=\varphi^{K_0}(\varphi^{R-m+1}(\varphi^{m-1+K-K_0}(0)))\,.
$$
It is easy to see that for every $j\geq m-1$, the factor $\varphi^{j}(0)$ contains each letter from $\A$, which follows from~\eqref{simpleParry}. We have $K\geq K_0$, thus $m-1+K-K_0\geq m-1$. Therefore, the factor $\varphi^{m-1+K-K_0}(0)$ contains each letter $\ell\in\A$. Hence
\begin{equation}\label{phiKR}
\varphi^{K+R}(0)=\varphi^{K_0}(w_0w_1w_2\cdots w_h)\,,
\end{equation}
where
$$
\{w_0,w_1,w_2,\ldots,w_h\}=\left\{\left.\varphi^{R-m+1}(\ell)\;\right|\;\ell\in\A\right\}\,.
$$
The definition of $R$ requires that the factor $\varphi^{R-m+1}(\ell)$ has the prefix $0$ for any $\ell\in\A$. As a result, each factor of type $\varphi^{K_0}(\varphi^{R-m+1}(\ell))$ has the prefix $\varphi^{K_0}(0)$. At the same time we know, with regard to the assumption $n\leq U_{K_0}$, that $\u_{[n]}$ is a prefix of $\varphi^{K_0}(0)$. To sum up, the words $\varphi^{R-m+1}(\ell)$ for $\ell\in\A$ have $\u_{[n]}$ as their prefixes. Now we can rewrite equation~\eqref{phiKR} as follows:
$$
\varphi^{K+R}(0)=\varphi^{K_0}(w_0)\varphi^{K_0}(w_1)\varphi^{K_0}(w_2)\cdots\varphi^{K_0}(w_h)=z_0z_1z_2\cdots z_h\,,
$$
where the factors $z_0,z_1,z_2,\ldots,z_h$ satisfy
\begin{equation}\label{zz}
\{z_0,z_1,z_2,\ldots,z_h\}=\left\{\left.\varphi^{K_0+R-m+1}(\ell)\;\right|\;\ell\in\A\right\},
\end{equation}
and, moreover, $\u_{[n]}$ is a prefix of $z_j$ for every $j\in\{0,1,\ldots,h\}$. This allows us to decompose
$$
\begin{array}{ccccccc}
\varphi^{K+R}(0)&=&z_0&z_1&z_2&\cdots&z_h \\
\u_{[n]}^{-1}\varphi^{K+R}(0)\u_{[n]}&=&\u_{[n]}^{-1}z_0\u_{[n]}&\u_{[n]}^{-1}z_1\u_{[n]}&\u_{[n]}^{-1}z_2\u_{[n]}&\cdots&\u_{[n]}^{-1}z_h\u_{[n]}
\end{array}
$$
Factors $z_j$ and $\tilde{z}_j=\u_{[n]}^{-1}z_j\u_{[n]}$ are abelian equivalent for every $j=0,1,\ldots h$, thus $\bigcup_{j=0}^h\left\{\bigl(\begin{smallmatrix}z_j\\ \tilde{z}_j\end{smallmatrix}\bigr)\right\}$ is an abelian co-decomposition of $\bigl(\begin{smallmatrix}\varphi^{K_0+R}(0)\\ \u_{[n]}^{-1}\varphi^{K+R}(0)\u_{[n]}\end{smallmatrix}\bigr)$. The ``maximal'' (i.e., having maximal number of elements) abelian co-decomposition of $\bigl(\begin{smallmatrix}\varphi^{K_0+R}(0)\\ \u_{[n]}^{-1}\varphi^{K+R}(0)\u_{[n]}\end{smallmatrix}\bigr)$ is obviously obtained as the union of the ``maximal'' abelian co-decompositions of $\binom{z_j}{\tilde{z}_j}$ for $j=0,1,\ldots h$, i.e.,
\begin{equation}\label{corr1}
\Codec\begin{pmatrix}\varphi^{K+R}(0)\\ \u_{[n]}^{-1}\varphi^{K+R}(0)\u_{[n]}\end{pmatrix}=\bigcup_{j=0}^h\Codec\begin{pmatrix}z_j\\ \u_{[n]}^{-1}z_j\u_{[n]}\end{pmatrix}\,.
\end{equation}
Finally, equation~\eqref{zz} gives the identity
\begin{equation}\label{corr2}
\bigcup_{j=0}^h\Codec\begin{pmatrix}z_j\\ \u_{[n]}^{-1}z_j\u_{[n]}\end{pmatrix}=\bigcup_{\ell\in\A}\Codec\begin{pmatrix}\varphi^{K_0+R-m+1}(\ell)\\ \u_{[n]}^{-1}\varphi^{K_0+R-m+1}(\ell)\u_{[n]}\end{pmatrix}\,.
\end{equation}
Combining equations~\eqref{corr1} and \eqref{corr2} one gets equation~\eqref{Codex_explicit}.
\end{proof}

The set $\Z_\u(n)$ together with the map $\Vect$ allow to determine the set of relative Parikh vectors corresponding to the number $n$.
Indeed, one can prove that~\cite[Prop.~3.8]{Tu13}
\begin{equation}\label{Pred(n)}
\Prel_\u(n)=\bigcup_{\begin{pmatrix}z\\ \tilde{z}\end{pmatrix}\in\Z_\u(n)}\Vect\begin{pmatrix}z\\ \tilde{z}\end{pmatrix}
\end{equation}
for any $n\in\N$. Consequently, if $\Z_\u(n)$ is known, it is a trivial task to calculate $\AC_\u(n)$ using the formula
\begin{equation}\label{ACP}
\AC_\u(n)=\#\bigcup_{\begin{pmatrix}z\\ \tilde{z}\end{pmatrix}\in\Z_\u(n)}\Vect\begin{pmatrix}z\\ \tilde{z}\end{pmatrix}\,,
\end{equation}
which follows immediately from equations~\eqref{ACrel} and \eqref{Pred(n)}.

\begin{example}\label{Z(1)}
Let us calculate $\Z_\t(1)$. We have $\mathbf{t}_{[1]}=0$. Since $1\leq T_0=1$ and $R=3-1+\min\{1,2,3,\ldots\}=3$, we shall use formula~\eqref{Z(n)} with $K+R=0+3=3$. Therefore, from~\eqref{Z(n)},
$$
\Z_\t(1)=\Codec\begin{pmatrix}\varphi_\t^3(0)\\ 0^{-1}\varphi_\t^3(0)0\end{pmatrix}\,.
$$
We have
$$
\begin{array}{ccc}
\varphi_\t^3(0)&=&01\:02\:01\:0 \\
0^{-1}\varphi_\t^3(0)0&=&10\:20\:10\:0
\end{array}
$$
whence we obtain
$$
\Z_\t(1)=\left\{\begin{pmatrix}0\\0\end{pmatrix},\begin{pmatrix}01\\10\end{pmatrix},\begin{pmatrix}02\\20\end{pmatrix}\right\}\,.
$$
Let us continue the example and demonstrate the application of equations~\eqref{Pred(n)} and \eqref{ACP}. We have
$$
\Vect\binom{0}{0}=\{(0,0,0)\},\; \Vect\binom{01}{10}=\{(-1,1,0),(0,0,0)\},\; \Vect\binom{02}{20}=\{(-1,0,1),(0,0,0)\}.
$$
Equation~\eqref{Pred(n)} then gives
$\Prel_\t(1)=\{(0,0,0),(-1,1,0),(-1,0,1)\}$, and hence, from
Eq.~\eqref{ACP} we get $\AC_\t(1)=3$. This calculation has an
illustrative purpose only -- the result $\AC_\t(1)=3$ could be of
course found much easier from equation~\eqref{AC}.
\end{example}

The essential point is that sets $\Z_\u(n)$ does not need to be calculated from definition~\eqref{Z(n)} for each $n\in\N$. We are going to present a recurrence relation that will allow us to express $\Z_\u(N)$ in terms of $\Z_\u(n)$ for a certain $n<N$.
Let $N\in\N$ be an integer with the normal $U$-representation $\langle N\rangle_U=d_k d_{k-1}\cdots d_1d_0$.
Let us take a $j\in[1,k]$, and split the sequence $d_k d_{k-1}\cdots d_1d_0$ into two parts in the way $d_k\cdots d_{j+1}$, $d_{j}\cdots d_0$. Note that both parts are valid normal $U$-representations of certain integers; we write them as
$n$ and $p$, as shown in equation~\eqref{Nnplk}.
\begin{equation}\label{Nnplk}
\langle N\rangle_U=\underbrace{d_k\cdots d_{j+1}}_{\langle n\rangle_U}\underbrace{d_{j}\cdots d_0}_{\langle p\rangle_U}
\end{equation}
Then we have the following result~\cite[Prop.~5.1]{Tu13}:

\begin{proposition}
Let the numbers $N,n,p$ and $k,j$ obey~\eqref{Nnplk}. If $\varphi^{j+1}(\tilde{z})$ has the prefix $\u_{[p]}$ for all $\begin{pmatrix}z\\ \tilde{z}\end{pmatrix}\in\Z_\u(n)$,
then $\Z_\u(N)$ can be calculated from $\Z_\u(n)$ using the formula
\begin{equation}\label{ZNMn}
\Z_\u(N)=\bigcup_{\begin{pmatrix}z\\ \tilde{z}\end{pmatrix}\in\Z_\u(n)}\Codec\begin{pmatrix}\varphi^{j+1}(z)\\ \u_{[p]}^{-1}\varphi^{j+1}(\tilde{z})\u_{[p]}\end{pmatrix}\,.
\end{equation}
\end{proposition}

For dealing with the Tribonacci word, the following corollary will come in handy. It is obtained straightforwardly from equation~\eqref{ZNMn} by choosing $j=0$ and $p=0$ or $p=1$.

\begin{corollary}\label{CorZ(N01)}
We have
\begin{align*}
\langle N\rangle_T=\langle n\rangle_T0 \qquad\Rightarrow\qquad & \Z_\t(N)=\bigcup_{\begin{pmatrix}z\\ \tilde{z}\end{pmatrix}\in\Z_\t(n)}\Codec\begin{pmatrix}\varphi_\t(z)\\ \varphi_\t(\tilde{z})\end{pmatrix}\,, \\
\langle N\rangle_T=\langle n\rangle_T1 \qquad\Rightarrow\qquad & \Z_\t(N)=\bigcup_{\begin{pmatrix}z\\ \tilde{z}\end{pmatrix}\in\Z_\t(n)}\Codec\begin{pmatrix}\varphi_\t(z)\\ 0^{-1}\varphi_\t(\tilde{z})0\end{pmatrix}\,.
\end{align*}
\end{corollary}

Since the abelian co-decompositions $\Codec\begin{pmatrix}\varphi_\t(z)\\ \varphi_\t(\tilde{z})\end{pmatrix}$ and $\Codec\begin{pmatrix}\varphi_\t(z)\\ 0^{-1}\varphi_\t(\tilde{z})0\end{pmatrix}$ will be often used throughout the paper, we introduce a shorthand for them. Let $z,\tilde{z}$ be factors of $\t$ satisfying $\Psi(z)=\Psi(\tilde{z})$, and let $\zeta$ stand for $\binom{z}{\tilde{z}}$. We define the symbols
\begin{equation}\label{D0D1}
\D_0(\zeta):=\Codec\begin{pmatrix}\varphi_\t(z)\\ \varphi_\t(\tilde{z})\end{pmatrix} \qquad\text{and}\qquad \D_1(\zeta):=\Codec\begin{pmatrix}\varphi_\t(z)\\ 0^{-1}\varphi_\t(\tilde{z})0\end{pmatrix}\,.
\end{equation}
Recall that the numbers $d_i$ in the normal $T$-representation~\eqref{Trep} attain the values $0$ and $1$ only. The statement of Corollary~\ref{CorZ(N01)} can be thus formulated as
\begin{equation}\label{n,delta}
\langle N\rangle_T=\langle n\rangle_T d \qquad\Rightarrow\qquad \Z_\t(N)=\bigcup_{\zeta\in\Z_\t(n)}\D_d(\zeta)\,.
\end{equation}
Formula~\eqref{n,delta} is crucial for us. It says how the set $\Z_\t(n)$ is transformed when a digit $d$ is added to $\langle n\rangle_T$. For this reason it will be sometimes refered to as the ``transformation formula''.

From now on we will focus on the Tribonacci word. Therefore, we can simplify the notation by dropping the subscript $\t$ from symbols $\AC_\t$, $\Prel_\t$, $\varphi_\t$, $\Z_\t$.

\section{Sets $\Z(n)$}\label{Sect.Z}

This section can be regarded as the core of the paper. Examining the structure of sets $\Z(n)$, we find out that only finitely many of them are mutually different. In other words, we show that there exist sets $\Z_1,\Z_2,\ldots,\Z_{M}$ such that for any $n\in\N$, the set $\Z(n)$ is equal to $\Z_q$ for a certain $q\in\{1,\ldots,M\}$. In addition, we express $\Z_1,\Z_2,\ldots,\Z_{M}$ explicitly.

\begin{lemma}\label{Z(N0),Z(N1)}
Let an $N\in\N$ have the representation $\langle N\rangle_T=1d_{k-1}\cdots d_1d_0$. Let $n$ be the number with the representation $\langle n\rangle_T=1d_{k-1}\cdots d_1$, i.e., $\langle N\rangle_T=\langle n\rangle_T d_0$. Then
\begin{equation}\label{Z(N)Z(n)}
\Z(N)=\bigcup_{\zeta\in\Z(n)}\D_0(\zeta) \qquad\text{or}\qquad \Z(N)=\bigcup_{\zeta\in\Z(n)}\D_1(\zeta)\,.
\end{equation}
\end{lemma}
Lemma~\ref{Z(N0),Z(N1)} is just a trivial consequence of the transformation formula~\eqref{n,delta}.
The main result of this section follows.

\begin{theorem}\label{Subsets}
There exist a $56$-element set
$$
\Z_{\mathrm{super}}=\{\zeta_1,\zeta_2,\zeta_3,\ldots,\zeta_{56}\}
$$
and $277$ of its subsets $\Z_1,\Z_2,\ldots,\Z_{277}\subset\Z_{\mathrm{super}}$ such that
\begin{equation}\label{Z1..277}
\left(\forall n\in\N\right)\,\left(\exists q\in\{1,2,\ldots,277\}\right)\,\left(\Z(n)=\Z_q\right)\,.
\end{equation}
\end{theorem}

\begin{proof}

We begin the proof by exploring $\Z(n)$ for $n$ having a $1$-digit representation, i.e., $\langle n\rangle_T=d_0$. Trivially, there is one single positive number having such representation, namely $n=1$.
We know from Example~\ref{Z(1)} that
$$
\Z(1)=\{\zeta_1,\zeta_2,\zeta_3\}
$$
for
$$
\zeta_1=\begin{pmatrix}0\\0\end{pmatrix},\quad\zeta_2=\begin{pmatrix}01\\10\end{pmatrix},\quad\zeta_3=\begin{pmatrix}02\\20\end{pmatrix}\,.
$$
From now on we let $\Z_1$ denote the set $\Z(1)$.

In the next step we proceed to exploring $\Z(N)$ for those $N$ having $2$-digit representations, i.e., $\langle N\rangle_T=1d_0$ for $d_0\in\{0,1\}$.
We apply Lemma~\ref{Z(N0),Z(N1)} with $k=1$. Writing $\langle N\rangle_T=1d_0$ in the form $\langle n\rangle_T d_0$ implies $\langle n\rangle_T=1$, and
hence $n=1$.
For such $N$ and $n$, equations~\eqref{Z(N)Z(n)} read
$$
\Z(N)=\bigcup_{\zeta\in\Z_1}\D_0(\zeta) \qquad\text{or}\qquad \Z(N)=\bigcup_{\zeta\in\Z_1}\D_1(\zeta)\,.
$$
Since $\Z_1=\{\zeta_1,\zeta_2,\zeta_3\}$, we need to find $\D_0(\zeta_j)$ and $\D_1(\zeta_j)$ for $j=1,2,3$, which is an easy task. Let us start with $\D_0(\zeta_j)$. With regard to~\eqref{D0D1}, we have to calculate
$$
\begin{array}{ccccccccccc}
\varphi(0)&=&0\:1 &\quad& \varphi(01)&=&0\:102 &\quad& \varphi(02)&=&0\:10 \\
\varphi(0)&=&0\:1 &\quad& \varphi(10)&=&0\:201 &\quad& \varphi(20)&=&0\:01
\end{array}
$$
Hence
\begin{equation}\label{D0zeta123}
\D_0(\zeta_1)=\{\zeta_1,\zeta_4\}\,,\quad \D_0(\zeta_2)=\{\zeta_1,\zeta_5\}\,,\quad \D_0(\zeta_3)=\{\zeta_1,\zeta_6\}\,,
\end{equation}
where $\zeta_1=\bigl(\begin{smallmatrix}0\\0\end{smallmatrix}\bigr)$ has been defined above and
$$
\zeta_4=\begin{pmatrix}1\\1\end{pmatrix}\,,\quad \zeta_5=\begin{pmatrix}102\\201\end{pmatrix}\,,\quad \zeta_6=\begin{pmatrix}10\\01\end{pmatrix}\,.
$$
A similar calculation leads to sets $\D_1(\zeta_j)$ for $j=1,2,3$; see~\eqref{D0D1}. We have
$$
\begin{array}{ccccccccccc}
\varphi(0)&=&01 &\quad& \varphi(01)&=&0102 &\quad& \varphi(02)&=&0\:1\:0 \\
0^{-1}\varphi(0)0&=&10 &\quad& 0^{-1}\varphi(10)0&=&2010 &\quad& 0^{-1}\varphi(20)0&=&0\:1\:0
\end{array}
$$
hence
\begin{equation}\label{D1zeta123}
\D_1(\zeta_1)=\{\zeta_2\}\,,\qquad \D_1(\zeta_2)=\{\zeta_7\}\,,\qquad \D_1(\zeta_3)=\{\zeta_1,\zeta_4\}
\end{equation}
for $\zeta_1,\zeta_2,\zeta_4$ defined above and
$$
\zeta_7=\begin{pmatrix}0102\\2010\end{pmatrix}\,.
$$
To sum up, if $\langle N\rangle_T=1d_0$, then
\begin{equation}\label{Z23}
\Z(N)=\{\zeta_1,\zeta_4,\zeta_5,\zeta_6\}:=\Z_2 \qquad\text{or}\qquad \Z(N)=\{\zeta_1,\zeta_2,\zeta_4,\zeta_7\}:=\Z_3\,.
\end{equation}
We will list sets $\Z_q$ in Table~\ref{Tab.Z}. The elements $\zeta_j$ together with $\D_0(\zeta_j)$ and $\D_1(\zeta_j)$ will be listed in Table~\ref{Tab.zeta}.

Now we can apply Lemma~\ref{Z(N0),Z(N1)} again, this time for $k=2$, in order to explore $\Z(N)$ for those $N$ having $3$-digit representations, $\langle N\rangle_T=1d_1d_0$. Any such $N$ can be written in the form $\langle n\rangle_T d_0$ for $\langle n\rangle_T=1d_1$ and $d_0\in\{0,1\}$. The representation of $n$ has two digits, and hence $\Z(n)=\Z_2$ or $\Z(n)=\Z_3$ due to the result of the previous step. We thus need to apply equations~\eqref{Z(N)Z(n)} for $\Z(n)=\Z_2$ and $\Z(n)=\Z_3$, which requires the knowledge of $\D_0(\zeta_j)$ and $\D_1(\zeta_j)$ for $j=1,\ldots,7$. Sets $\D_0(\zeta_j)$ and $\D_1(\zeta_j)$ for $j=1,2,3$ are already known; see equations~\eqref{D0zeta123},~\eqref{D1zeta123}. Concerning $j=4,\ldots,7$, a short calculation gives
\begin{equation}\label{D01zeta4..7}
\begin{array}{llll}
\D_0(\zeta_4)=\{\zeta_1,\zeta_8\},\; & \D_0(\zeta_5)=\{\zeta_1,\zeta_9\}, & \D_0(\zeta_6)=\{\zeta_1,\zeta_{10}\},\; & \D_0(\zeta_7)=\{\zeta_1,\zeta_6,\zeta_{11}\}, \\
\D_1(\zeta_4)=\{\zeta_3\}, & \D_1(\zeta_5)=\{\zeta_1,\zeta_{10}\},\; & \D_1(\zeta_6)=\{\zeta_{12}\}, & \D_1(\zeta_7)=\{\zeta_1,\zeta_4,\zeta_8\}
\end{array}
\end{equation}
for
$$
\zeta_8=\begin{pmatrix}2\\2\end{pmatrix}\,,\quad \zeta_9=\begin{pmatrix}2010\\0102\end{pmatrix}\,,\quad \zeta_{10}=\begin{pmatrix}201\\102\end{pmatrix}\,,\quad \zeta_{11}=\begin{pmatrix}20\\02\end{pmatrix}\,,\quad \zeta_{12}=\begin{pmatrix}0201\\1020\end{pmatrix}\,.
$$
Once we substitute $\Z(n)=\Z_2$ and $\Z(n)=\Z_3$ into equation~\eqref{Z(N)Z(n)} and use the known structure of $\Z_2,\Z_3$ (cf.~\eqref{Z23}) together with equations~\eqref{D0zeta123},~\eqref{D1zeta123} and~\eqref{D01zeta4..7}, we get
\begin{align*}
\Z(n)&=\Z_2 \quad\Rightarrow\quad \\
& \Z(N)=\{\zeta_1,\zeta_4,\zeta_8,\zeta_9,\zeta_{10}\}=:\Z_4 \quad\text{or}\quad \Z(N)=\{\zeta_1,\zeta_2,\zeta_3,\zeta_{10},\zeta_{12}\}=:\Z_6\,; \\
\Z(n)&=\Z_3 \quad\Rightarrow\quad \\
& \Z(N)=\{\zeta_1,\zeta_4,\zeta_5,\zeta_6,\zeta_8,\zeta_{11}\}=:\Z_5 \quad\text{or}\quad \Z(N)=\{\zeta_1,\zeta_2,\zeta_3,\zeta_4,\zeta_7,\zeta_8\}=:\Z_7\,.
\end{align*}

\begin{table}
\begin{center}
\begin{tabular}{|c|l|}
\hline
$q$ & $\{j\,|\,\zeta_j\in\Z_q\}$ \\
\hline
 1  &  1 ; 2 ; 3  \\
 2  &  1 ; 4 ; 5 ; 6  \\
 3  &  1 ; 2 ; 4 ; 7  \\
 4  &  1 ; 4 ; 8 ; 9 ; 10  \\
 5  &  1 ; 4 ; 5 ; 6 ; 8 ; 11  \\
 6  &  1 ; 2 ; 3 ; 10 ; 12  \\
 7  &  1 ; 2 ; 3 ; 4 ; 7 ; 8  \\
 8  &  1 ; 2 ; 4 ; 8 ; 9 ; 10  \\
 9  &  1 ; 4 ; 5 ; 6 ; 7 ; 13  \\
 10  &  1 ; 2 ; 3 ; 14 ; 15  \\
 11  &  1 ; 2 ; 3 ; 10 ; 12 ; 16  \\
 12  &  1 ; 2 ; 4 ; 7 ; 15 ; 17  \\
 \vdots  &    \\
 277  &  1 ; 2 ; 4 ; 7 ; 15 ; 17 ; 22 ; 23 ; 24 ; 36 ; 43 ; 50  \\
\hline
\end{tabular}
\end{center}
\caption{Structure of sets $\Z_q$.}
\label{Tab.Z}
\end{table}

\begin{table}
\begin{center}
\begin{tabular}{|c|l|l|l|l|}
\hline
$j$ & $\zeta_j$ & $\{i\,|\,\zeta_i\in\D_0(\zeta_j)\}$ & $\{i\,|\,\zeta_i\in\D_1(\zeta_j)\}$ & $\{i\,|\,\psi^{\mathrm{r}}_i\in\Vect(\zeta_j)\}$ \\[2pt]
\hline
1 & $\bigl(\begin{smallmatrix}0\\0\end{smallmatrix}\bigr)$ & 1;4 & 2 & 0 \\[2pt]
2 & $\bigl(\begin{smallmatrix}01\\10\end{smallmatrix}\bigr)$ & 1;5 & 7 & 0;1 \\[2pt]
3 & $\bigl(\begin{smallmatrix}02\\20\end{smallmatrix}\bigr)$ & 1;6 & 1;4 & 0;2 \\[2pt]
4 & $\bigl(\begin{smallmatrix}1\\1\end{smallmatrix}\bigr)$ & 1;8 & 3 & 0 \\[2pt]
5 & $\bigl(\begin{smallmatrix}102\\201\end{smallmatrix}\bigr)$ & 1;9 & 1;10 & 0;3 \\[2pt]
6 & $\bigl(\begin{smallmatrix}10\\01\end{smallmatrix}\bigr)$ & 1;10 & 12 & 0;4 \\[2pt]
7 & $\bigl(\begin{smallmatrix}0102\\2010\end{smallmatrix}\bigr)$ & 1;6;11 & 1;4;8 & 0;2;3 \\[2pt]
8 & $\bigl(\begin{smallmatrix}2\\2\end{smallmatrix}\bigr)$ & 1 & 1 & 0 \\[2pt]
9 & $\bigl(\begin{smallmatrix}2010\\0102\end{smallmatrix}\bigr)$ & 1;2;3 & 14 & 0;5;6 \\[2pt]
10 & $\bigl(\begin{smallmatrix}201\\102\end{smallmatrix}\bigr)$ & 1;7 & 15 & 0;6 \\[2pt]
11 & $\bigl(\begin{smallmatrix}20\\02\end{smallmatrix}\bigr)$ & 1;2 & 16 & 0;5 \\[2pt]
12 & $\bigl(\begin{smallmatrix}0201\\1020\end{smallmatrix}\bigr)$ & 1;13 & 17 & 0;1;6 \\[2pt]
\vdots &  &  &  &  \\
54 & $\bigl(\begin{smallmatrix}0010201010201\\1020101020100\end{smallmatrix}\bigr)$ & 1;18;19;56 & 1;22;23;41 & 0;1;2 \\[2pt]
55 & $\bigl(\begin{smallmatrix}00102010201\\10201020100\end{smallmatrix}\bigr)$ & 1;18;19;37 & 1;13;22;23 & 0;1;2 \\[2pt]
56 & $\bigl(\begin{smallmatrix}10010201020\\02010201001\end{smallmatrix}\bigr)$ & 1;25;26;52 & 1;27;28;29 & 0;3;4 \\[2pt]
\hline
\end{tabular}
\end{center}
\caption{Possible elements of sets $\Z(n)$. The table shows the elements $\zeta_j$ and the structure of sets $\D_0(\zeta_j),\D_1(\zeta_j)$. It also shows the structure of sets $\Vect(\zeta_j)$, cf.\ Sect.~\ref{Sect.Range}.}
\label{Tab.zeta}
\end{table}

The calculation continues in the same way for $k=3$. For each $q\in\{4,5,6,7\}$, we substitute the set $\Z_q$ into~\eqref{Z(N)Z(n)} for $\Z(n)$. Evaluation of $\Z(N)$ requires sets $\D_0(\zeta_j)$ and $\D_1(\zeta_j)$ for $j=1,\ldots,12$. Some of them are known from our previous calculation ($j=1,\ldots,7$), others can be calculated similarly now ($j=8,\ldots,12$). In this way we find eight sets $\Z(N)$, namely
\begin{equation}\label{eight}
\begin{array}{ccccc}
\Z(n)=\Z_4 & \quad\Rightarrow\quad & \Z(N)=\Z_7 & \quad\text{or}\quad & \Z(N)=\Z_{10}\,, \\
\Z(n)=\Z_5 & \quad\Rightarrow\quad & \Z(N)=\Z_8 & \quad\text{or}\quad & \Z(N)=\Z_{11}\,, \\
\Z(n)=\Z_6 & \quad\Rightarrow\quad & \Z(N)=\Z_9 & \quad\text{or}\quad & \Z(N)=\Z_{12}\,, \\
\Z(n)=\Z_7 & \quad\Rightarrow\quad & \Z(N)=\Z_5 & \quad\text{or}\quad & \Z(N)=\Z_7\,.
\end{array}
\end{equation}
Their structure is shown in Table~\ref{Tab.Z}.

In the next step we use Lemma~\ref{Z(N0),Z(N1)} again. The eight sets found in~\eqref{eight} should be plugged in equations~\eqref{Z(N)Z(n)}, which will lead to sets $\Z(N)$ for $N$ having $5$-digit representations ($k=4$). Note, however, that the eight sets obtained in~\eqref{eight} are of two types:
\begin{itemize}
\item three of them ($\Z_7$ occurring $2$ times and $\Z_5$) had been found in earlier steps;
\item five of them ($\Z_8,\Z_9,\Z_{10},\Z_{11},\Z_{12}$) are ``new'' -- they have just appeared in the calculation for the first time.
\end{itemize}
Since sets $\Z_5$ and $\Z_7$ have been already explored, there is no need to apply equations~\eqref{Z(N)Z(n)} on them again. It suffices to use formula~\eqref{Z(N)Z(n)} only with $\Z(n)=\Z_q$ for such $\Z_q$ that promise potential \emph{new} results, i.e., $q=8,9,10,11,12$. Henceforth we will proceed similarly -- we will always put aside those sets $\Z_q$ that reappear after having been explored in earlier stages of the procedure, and queue only the ``new'' ones for further use in~\eqref{Z(N)Z(n)}.

\begin{table}
\begin{center}
\begin{tabular}{|c|l|l||c|l|l|}
\hline
$k$ & new $\Z_q$ found & new $\zeta_j$ found & $k$ & new $\Z_q$ found & new $\zeta_j$ found \\
\hline
0 & $\Z_1$ & $\zeta_1,\zeta_2,\zeta_3$ & 15 & $\Z_{201},\ldots,\Z_{221}$ & $\zeta_{56}$ \\
1 & $\Z_2,\Z_3$ & $\zeta_4,\zeta_5,\zeta_6,\zeta_7$ & 16 & $\Z_{222},\ldots,\Z_{245}$ & none \\ \cline{6-6}
2 & $\Z_4,\Z_5,\Z_6,\Z_7$ & $\zeta_8,\zeta_9,\zeta_{10},\zeta_{11},\zeta_{12}$ & 17 & $\Z_{246},\ldots,\Z_{260}$ & \\
3 & $\Z_8,\Z_9,\Z_{10},\Z_{11},\Z_{12}$ & $\zeta_{13},\zeta_{14},\zeta_{15},\zeta_{16},\zeta_{17}$ & 18 & $\Z_{261},\ldots,\Z_{271}$ & \\
4 & $\Z_{13},\ldots,\Z_{21}$ & $\zeta_{18},\ldots,\zeta_{24}$ & 19 & $\Z_{272},\Z_{273}$ & \\
5 & $\Z_{22},\ldots,\Z_{32}$ & $\zeta_{25},\ldots,\zeta_{30}$ & 20 & $\Z_{274},\Z_{275}$ & \\
\vdots &  &  & 21 & $\Z_{276},\Z_{277}$ & \\
14 & $\Z_{179},\ldots,\Z_{200}$ & $\zeta_{54},\zeta_{55}$ & 22 & none & \\
\hline
\end{tabular}
\end{center}
\caption{Progress of the calculation. Sets $\Z_q$ expressed in terms of $\zeta_j$ can be found in Tab.~\ref{Tab.Z}. Elements $\zeta_j$ are listed in Tab.~\ref{Tab.zeta}.}
\label{Tab.proc}
\end{table}
The progress of the calculation is illustrated with Table~\ref{Tab.proc}. We see that when $k$ reaches the value $22$, no new set $\Z_q$ is found. In other words, all sets $\Z(N)$ obtained by equations~\eqref{Z(N)Z(n)} for $k=22$ have been already found (and explored) earlier. The search is then completed.
We conclude: there exist altogether $277$ sets $\Z_1,\ldots,\Z_{277}$ such that for any $n\in\N$, we have $\Z(n)=\Z_q$ for a certain $q\in\{1,\ldots,277\}$.

Since every $\Z_q$ consists of elements $\zeta_j$ for $j=1,\ldots,56$ (note that new elements $\zeta_j$ stop appearing already at $k=16$), the set $\Z_{\mathrm{super}}:=\{\zeta_1,\zeta_2,\zeta_3,\ldots,\zeta_{55},\zeta_{56}\}$ obviously satisfies $\Z_q\subset\Z_{\mathrm{super}}$ for all $q=1,\ldots,277$.

\end{proof}

The search for sets $\Z_q$ and elements $\zeta_j$ can be in principle carried out completely using pen and paper, but since the procedure is lengthy and cumbersome, it is more convenient to use a computer, which also helps to avoid mistakes. Below we provide a pseudocode of the algorithm. Note that some variables used in the pseudocode are of special types, such as sets or pairs of strings.

\begin{algorithmic}[1]
\State calculate $\Z(1)$ using formula~\eqref{Z(n)} for $R=3$ and $K=0$
\State $a\gets$ number of elements of $\Z(1)$
\State $Z[1]\gets$ the set $\{1,\ldots,a\}$
\State $M\gets 1$
\State $M\_old\gets 0$
\State $M\_new\gets 1$
\For{$j=1,\ldots,a$}
   \State $zeta[j]\gets$ $j$-th element of $\Z(1)$
\EndFor
\State $a\_old\gets 0$
\State $a\_new\gets a$

\While{$M\_old<M\_new$}

   \For{$d=0,1$}
       \For{$j=a\_old+1,\ldots,a\_new$}
           \State $z\gets zeta[j]$
           \State $D\gets \emptyset$
           \State calculate $\D_d(z)$ using formula~\eqref{D0D1}
           \For{every $x\in\D_d(z)$}
               \If {an $i\in\{1,\ldots,a\}$ satisfies $x=zeta[i]$}
                  \State $D\gets D\cup\{i\}$
               \Else
                  \State $D\gets D\cup\{a+1\}$
                  \State $zeta[a+1]\gets x$
                  \State $a\gets a+1$
               \EndIf
           \EndFor
           \State $D[d,j]\gets D$
       \EndFor
   \EndFor
   \State $a\_old\gets a\_new$
   \State $a\_new\gets a$

   \For{$q=M\_old+1,\ldots,M\_new$}
		   \For{$d=0,1$}
		       \State $D\gets \emptyset$
		       \For{every $j\in Z[q]$}
               \State $D\gets D\cup D[d,j]$
			     \EndFor
	         \If {a $j\in\{1,\ldots,M\}$ satisfies $D=Z[j]$}
			        \State $delta[q,d]\gets j$
			     \Else
				      \State $delta[q,d]\gets M+1$
	            \State $Z[M+1]\gets D$
	            \State $M\gets M+1$
           \EndIf
       \EndFor
   \EndFor
   \State $M\_old\gets M\_new$
   \State $M\_new\gets M$				

\EndWhile
\end{algorithmic}

When the algorithm ends, final values of variables have the meanings as follows.
\begin{itemize}
\item $a$: cardinality of $\Z_{\mathrm{super}}$;
\item $zeta[j]$ for $j=1,\ldots,a$: $\zeta_j$ (cf.\ Table~\ref{Tab.zeta});
\item $D[d,j]$ for $d=0,1$, $j=1,\ldots,a$: the set $\{i\,|\,\zeta_i\in\D_d(\zeta_j)\}$ (cf.\ Table~\ref{Tab.zeta});
\item $M$: cardinality of $\{\Z(n)\;|\;n\in\N\}$ (cf.\ Table~\ref{Tab.Z});
\item $Z[q]$ for $q=1,\ldots,M$: the set $\{j\;|\;\zeta_j\in\Z_q\}$ (cf.\ Table~\ref{Tab.Z});
\item $delta[q,d]$ for $q=1,\ldots,M$,  $d=0,1$: $\delta(q,d)$ (cf.\ Sect.~\ref{Sect.Eval.AC}, Table~\ref{Tab.sigma}).
\end{itemize}

\section{Range of the abelian complexity}\label{Sect.Range}

The aim of this section is twofold:
\begin{itemize}
\item We illustrate that the result of Section~\ref{Sect.Z} allows to examine the image of the abelian complexity function, as well as the balance bound of $\t$.
\item We express $\Vect(\zeta_j)$ for $j=1,\ldots,56$, and we associate each $\Z_q$ for $q=1,\ldots,277$ with a certain value of the abelian complexity function. This result will be useful later in Section~\ref{Sect.Eval.AC} for an explicit evaluation of $\AC(n)$.
\end{itemize}

Recall that for any $n\in\N$, the value of the abelian complexity function $\AC(n)$ can be determined by formula~\eqref{ACP} if the set $\Z(n)$ is known. We have not found an explicit relation between $n$ and $\Z(n)$ yet, and thus we cannot use~\eqref{ACP} as it is. Nevertheless, we have demonstrated in Section~\ref{Sect.Z} that $\Z(n)=\Z_q$ for a certain $q\in\{1,\ldots,277\}$, which already gives us the possibility to restrict the range of $\AC(n)$.

The use of formula~\eqref{ACP} requires $\Vect(\zeta)$ for $\zeta\in\Z(n)$. Therefore, we shall express $\Vect(\zeta_j)$ for $j=1,\ldots,56$.
If equation~\eqref{P} is applied on $\zeta_1$, it gives
\begin{equation*}
\begin{split}
\Vect(\zeta_1)&=\Vect\begin{pmatrix}0\\0\end{pmatrix}=\left\{\Psi(s)-\Psi(r)\;\left|\;\text{$r$ is a prefix of $0$, $s$ is a prefix of $0$, $|s|=|r|$}\right.\right\}\\
&=\{(0,0,0)\}.
\end{split}
\end{equation*}
Similarly, for $\zeta_2$ and $\zeta_3$ we get
\begin{equation*}
\begin{split}
\Vect(\zeta_2)&=\Vect\begin{pmatrix}01\\10\end{pmatrix}=\left\{\Psi(s)-\Psi(r)\,\left|\,\text{$r$ is a prefix of $01$, $s$ is a prefix of $10$, $|s|=|r|$}\right.\right\}\\
&=\{(-1,1,0),(0,0,0)\},
\end{split}
\end{equation*}
\begin{equation*}
\begin{split}
\Vect(\zeta_3)&=\Vect\begin{pmatrix}02\\20\end{pmatrix}=\left\{\Psi(s)-\Psi(r)\,\left|\,\text{$r$ is a prefix of $02$, $s$ is a prefix of $20$, $|s|=|r|$}\right.\right\}\\
&=\{(-1,0,1),(0,0,0)\}.
\end{split}
\end{equation*}
Performing the same calculation for every value of $j$, we find
that for all $j=1,\ldots,56$, we have $\Vect(\zeta_j)\subset\{\psi^{\mathrm{r}}_0,\psi^{\mathrm{r}}_1,\ldots,\psi^{\mathrm{r}}_8\}$, where
$$
\begin{array}{lclcl}
\psi^{\mathrm{r}}_0=(0,0,0)\,; &\quad& \psi^{\mathrm{r}}_1=(-1,1,0)\,; &\quad& \psi^{\mathrm{r}}_2=(-1,0,1)\,; \\
\psi^{\mathrm{r}}_3=(0,-1,1)\,; &\quad& \psi^{\mathrm{r}}_4=(1,-1,0)\,; &\quad& \psi^{\mathrm{r}}_5=(1,0,-1)\,; \\
\psi^{\mathrm{r}}_6=(0,1,-1)\,; &\quad& \psi^{\mathrm{r}}_7=(-1,2,-1)\,; &\quad& \psi^{\mathrm{r}}_8=(-1,-1,2)\,.
\end{array}
$$
The structure of sets $\Vect(\zeta_j)$ in terms of $\psi^{\mathrm{r}}_0,\ldots,\psi^{\mathrm{r}}_8$ is partly shown in Table~\ref{Tab.zeta}.

According to~\eqref{Z1..277}, for every $n\in\N$ there is a $q\in\{1,\ldots,277\}$ such that $\Z(n)=\Z_q$. By virtue of equation~\eqref{ACP}, we have
\begin{equation}\label{ACtau}
\Z(n)=\Z_q \qquad\Rightarrow\qquad \AC(n)=\#\bigcup_{\zeta\in\Z_q}\Vect(\zeta)\,.
\end{equation}
The expression $\#\bigcup_{\zeta\in\Z_q}\Vect(\zeta)$ depends only on $q$. Therefore, it is convenient to introduce the function
\begin{equation}\label{Aj}
\tau(q):=\#\bigcup_{\zeta\in\Z_q}\Vect(\zeta)
\end{equation}
for $q=1,\ldots,277$.
Evaluation of $\tau$ is a straightforward task, consisting of combining data from Tables~\ref{Tab.Z} and \ref{Tab.zeta}:
\begin{enumerate}
\item Table~\ref{Tab.Z} shows the structure of $\Z_q$ in terms of $\zeta_j$.
\item Table~\ref{Tab.zeta} contains $\Vect(\zeta_j)$ for $\zeta_j$.
\end{enumerate}
For example, for $q=1$ we obtain
$$
\bigcup_{\zeta\in\Z_1}\Vect(\zeta)\stackrel{(a)}{=}\Vect(\zeta_1)\cup \Vect(\zeta_2)\cup \Vect(\zeta_3)\stackrel{(b)}{=}\left\{\psi^{\mathrm{r}}_0\right\}\cup\left\{\psi^{\mathrm{r}}_0,\psi^{\mathrm{r}}_1\right\}\cup\left\{\psi^{\mathrm{r}}_0,\psi^{\mathrm{r}}_2\right\}=\left\{\psi^{\mathrm{r}}_0,\psi^{\mathrm{r}}_1,\psi^{\mathrm{r}}_2\right\},
$$
where we used data from Table~\ref{Tab.Z} at (a) and data from Table~\ref{Tab.zeta} at (b).
Hence $\tau(1)=\#\bigcup_{\zeta\in\Z_1}\Vect(\zeta)=3$.
We proceed similarly for $q=2,\ldots,277$.
\begin{table}
\begin{center}
\begin{tabular}{|c|c||c|c||c|c||c|c||c|c||c|c||c|c|}
\hline
$q$ & $\tau(q)$ & $q$ & $\tau(q)$ & $q$ & $\tau(q)$ & $q$ & $\tau(q)$ & $q$ & $\tau(q)$ & $q$ & $\tau(q)$ & $q$ & $\tau(q)$ \\
\hline
 1  &  3  &  7  &  4  &  13  &  4  &  19  &  4  &  25  &  4  &  31  &  4  &  37  &  4  \\
 2  &  3  &  8  &  4  &  14  &  4  &  20  &  4  &  26  &  4  &  32  &  4  &  38  &  4  \\
 3  &  4  &  9  &  4  &  15  &  3  &  21  &  4  &  27  &  4  &  33  &  5  &  39  &  4  \\
 4  &  3  &  10  &  3  &  16  &  4  &  22  &  5  &  28  &  4  &  34  &  5  &  40  &  4  \\
 5  &  4  &  11  &  4  &  17  &  4  &  23  &  5  &  29  &  4  &  35  &  4  &  41  &  3  \\
 6  &  4  &  12  &  4  &  18  &  4  &  24  &  3  &  30  &  4  &  36  &  4  &  42  &  4  \\
\hline
\end{tabular}
\end{center}
\caption{The function $\tau$ (partial list for $q=1,\ldots,42$).}
\label{Tab.C}
\end{table}
The values $\tau(q)$ obtained by the calculation are listed in Table~\ref{Tab.C}. The list is only partial by reason of saving space, however, in the full list we could see that $\tau(q)\in\{3,4,5,6,7\}$ for all $q\in\{1,\ldots,277\}$. Consequently, $\AC(n)\in\{3,4,5,6,7\}$ for all $n\in\N$.

Moreover, one can show similarly that for any $q=1,\ldots,277$, the absolute values of the elements of the difference vector $\psi^{\mathrm{r}}_i-\psi^{\mathrm{r}}_j$ for $\psi^{\mathrm{r}}_i,\psi^{\mathrm{r}}_j\in\bigcup_{\zeta\in\Z_q}\Vect(\zeta)$ do not exceed $2$.
This fact implies that the Tribonacci word is $2$-balanced.
However, this bound as well as the result $\AC(n)\in\{3,4,5,6,7\}$ are not new; they have been derived already in~\cite{RSZ}.

\section{Evaluation of the abelian complexity}\label{Sect.Eval.AC}

As we have noticed at the beginning of Section~\ref{Sect.Range}, there is one last step to do before we can calculate $\AC(n)$ explicitly. Namely, we need to associate each $n\in\N$ with the number $q\in\{1,\ldots,277\}$ such that $\Z(n)=\Z_q$. The task will be solved in this section.

We introduce the map (the ``transition function'') $\delta(q,d)$ for $q\in\{1,\ldots,277\}$ and $d\in\{0,1\}$ by the formula
\begin{equation}\label{tau}
\delta(q,d)=\ell \quad\Leftrightarrow\quad \bigcup_{\zeta\in\Z_q}\D_d(\zeta)=\Z_\ell\,.
\end{equation}
Then we extend its domain to the value $q=0$ by setting
\begin{equation}\label{tauq=0}
\delta(0,0)=0, \qquad \delta(0,1)=1.
\end{equation}
Values $\delta(q,d)$ for $q=1,\ldots,277$ can be naturally gathered in the course of the calculations carried out in the proof of Theorem~\ref{Subsets}. For example, we have $\bigcup_{\zeta\in\Z_1}\D_0(\zeta)=\Z_2$ and $\bigcup_{\zeta\in\Z_1}\D_1(\zeta)=\Z_3$; hence $\delta(1,0)=2$ and $\delta(1,1)=3$. The values can be then tabulated; see Table~\ref{Tab.sigma}.

\begin{table}
\begin{center}
\begin{tabular}{|c|cc||c|cc||c|cc||c|cc|}
\hline
$q$ & $\delta(q,0)$ & $\delta(q,1)$ & $q$ & $\delta(q,0)$ & $\delta(q,1)$ & $q$ & $\delta(q,0)$ & $\delta(q,1)$ & $q$ & $\delta(q,0)$ & $\delta(q,1)$ \\
\hline
 1  &  2  &  3  &  11  &  16  &  21  &  21  &  17  &  13  &  31  &  22  &  23  \\
 2  &  4  &  6  &  12  &  17  &  13  &  22  &  33  &  39  &  32  &  17  &  13  \\
 3  &  5  &  7  &  13  &  22  &  23  &  23  &  34  &  40  &  33  &  47  &  54  \\
 4  &  7  &  10  &  14  &  23  &  28  &  24  &  7  &  41  &  34  &  48  &  55  \\
 5  &  8  &  11  &  15  &  24  &  29  &  25  &  13  &  42  &  35  &  13  &  56  \\
 6  &  9  &  12  &  16  &  14  &  19  &  26  &  35  &  43  &  36  &  49  &  57  \\
 7  &  5  &  7  &  17  &  25  &  30  &  27  &  14  &  19  &  37  &  14  &  19  \\
 8  &  13  &  18  &  18  &  26  &  31  &  28  &  36  &  44  &  38  &  14  &  19  \\
 9  &  14  &  19  &  19  &  27  &  32  &  29  &  37  &  45  &  39  &  50  &  58  \\
 10  &  15  &  20  &  20  &  22  &  23  &  30  &  38  &  46  &  40  &  17  &  13  \\
\hline
\end{tabular}
\end{center}
\caption{Values of the transition function $\delta(q,d)$ (partial list for $q=1,\ldots,40$).}
\label{Tab.sigma}
\end{table}

The transition function $\delta(q,d)$ allows to determine $\Z(n)$ for any $n\in\N$, as we demonstrate in the following proposition.

\begin{proposition}\label{Zndelta}
If $\langle n \rangle_T =d_k d_{k-1}\cdots d_1d_0$, then
\begin{equation}\label{Zntau}
\Z(n)=\Z_q \qquad\text{for}\quad q=\delta(0,d_k d_{k-1}\cdots d_1d_0)\,,
\end{equation}
where the symbol $\delta(0,d_k d_{k-1}\cdots d_1d_0)$ is a shorthand for $\delta(\delta(\cdots\delta(\delta(0,d_k),d_{k-1})\cdots,d_1),d_0)$ (cf.\ Sect.~\ref{Sect.Preliminaries}).
\end{proposition}

\begin{proof}
At first, let us assume that $d_k=1$; the case $d_k=0$ will be treated in the end. We prove the statement by induction on $k$.

I.\quad Let $k=0$. Then $\langle n \rangle_T =1$, i.e., $n=1$. We have $\delta(0,d_0)=\delta(0,1)=1$ from Eq.~\eqref{tauq=0}, thus formula~\eqref{Zntau} gives $\Z(1)=\Z_1$. This is obviously true with regard to the definition of $\Z_1$.

II.\quad Let the formula~\eqref{Zntau} be valid for numbers with normal $T$-representations having at most $k$ digits. Let a number $n$ have a $(k+1)$-digit representation $\langle n \rangle_T =d_k d_{k-1}\cdots d_1d_0$. The formula~\eqref{Zntau} is valid for $\langle n' \rangle_T =d_k d_{k-1}\cdots d_1$ (because $\langle n'\rangle_T$ has only $k$ digits), and hence $\Z(n')=\Z_j$ for $j=\delta(0,d_k d_{k-1}\cdots,d_1)$.
From Eq.~\eqref{n,delta} we have $\Z(n)=\bigcup_{\zeta\in\Z_j}\D_{d_0}(\zeta)$. Consequently, with regard to the definition of the transition function $\delta$ (cf.~\eqref{tau}), the equation $\Z(n)=\Z_q$ holds for $q=\delta(j,d_0)$. Combining these two facts, we obtain
$$
\Z(n)=\Z_q \quad\text{for}\quad q=\delta(\delta(0,d_k d_{k-1}\cdots,d_1),d_0)=\delta(0,d_k d_{k-1}\cdots d_1d_0)\,,
$$
which completes the inductive step.

It suffices to deal with the case $d_k=0$, i.e., $\langle n\rangle_T=00\cdots0d_\ell d_{\ell-1}\cdots d_0$ for an $\ell<k$ and $d_\ell=1$. The initial block of zeros in the representation can be omitted, i.e., we can write $\langle n\rangle_T=d_\ell d_{\ell-1}\cdots d_0$ for $d_\ell=1$. Then, due to previous considerations, $\Z(n)=\Z_q$ for $q=\delta(0,d_\ell d_{\ell-1}\cdots,d_0)$. At the same time, since $\delta(0,0)=0$ from Eq.~\eqref{tauq=0}, we trivially get
$\delta(0,0\cdots0)=0$. Hence
\begin{equation*}
\begin{split}
q&=\delta(0,d_\ell d_{\ell-1}\cdots,d_0)=\delta(\delta(0,0\cdots0),d_\ell d_{\ell-1}\cdots,d_0)=\delta(0,0\cdots0d_\ell d_{\ell-1}\cdots d_1d_0) \\
&=\delta(0,d_k d_{k-1}\cdots d_{\ell+1}d_\ell d_{\ell-1}\cdots d_1d_0)\,.
\end{split}
\end{equation*}
\end{proof}

Proposition~\ref{Zndelta} associates any $n\in\N$ with a $q\in\{1,\ldots,277\}$ such that $\Z(n)=\Z_q$. Formula~\eqref{ACtau} together with~\eqref{Aj} then gives the value of the abelian complexity as $\tau(q)$, where $\tau$ is the function tabulated in Table~\ref{Tab.C}.
Combining these two results, we get
\begin{theorem}\label{MainThm}
For any $n\in\N$ we have
\begin{equation}\label{ACfinal}
\AC(n)=\tau(\delta(0,\langle n \rangle_T))\,.
\end{equation}
\end{theorem}
Obviously, Theorem~\ref{MainThm} together with Tables~\ref{Tab.sigma} and~\ref{Tab.C} allows to determine $\AC(n)$ for any $n\in\N$.

Equation~\eqref{ACfinal} implies that $(\AC(n))_{n\in\N}$ is a $T$-automatic sequence. It is generated by the DFAO $(Q,\Sigma,\delta,q_0,\Delta,\tau)$, where
\begin{itemize}
\item $Q=\{0,1,\ldots,277\}$ is the set of states;
\item $\Sigma=\{0,1\}$ is the input alphabet;
\item $\delta$ is the state-transition function, cf.\ Table~\ref{Tab.sigma};
\item $q_0=0$ is the initial state;
\item $\Delta=\{3,4,5,6,7\}$ is the output alphabet;
\item $\tau$ is the output function, specified in Table~\ref{Tab.C}.
\end{itemize}

The number of steps needed for evaluating $\AC(n)$ is proportional to the number of digits in the normal $T$-representation $\langle n \rangle_T$. 
The $T_j$ grow roughly exponentially ($T_j\approx\beta^j$, where $\beta\approx1.84$ is the root of $x^3=x^2+x+1$), and hence the number of digits is $k\sim\log n$. The evaluation of $\AC(n)$ using formula~\eqref{ACfinal} thus needs $\mathcal{O}(\log n)$ operations. It is therefore fast even for large values of $n$.

\begin{example}\label{Ex.2013}
Let us calculate $\AC(n)$ for $n=2013$. Since
$$
(T_j)_{j=3}^{\infty}=(1,2,4,7,13,24,44,81,149,274,504,927,1705,3136,\ldots)\,,
$$
we have $\langle2013\rangle_T=1001000101011$, cf.~\eqref{TribExp}. Starting from the initial state $0$ and applying the state-transition function $\delta$ step by step, we obtain
$$
0\stackrel{1}{\rightarrow}1\stackrel{0}{\rightarrow}2\stackrel{0}{\rightarrow}4\stackrel{1}{\rightarrow}10\stackrel{0}{\rightarrow}15\stackrel{0}{\rightarrow}24\stackrel{0}{\rightarrow}7\stackrel{1}{\rightarrow}7\stackrel{0}{\rightarrow}5\stackrel{1}{\rightarrow}11\stackrel{0}{\rightarrow}16\stackrel{1}{\rightarrow}19\stackrel{1}{\rightarrow}32\,,
$$
thus $\delta(0,1001000101011)=32$. Hence $\AC(2013)=\tau(32)=4$.
\end{example}

\section{Reduction of states of the automaton}\label{Sect.Simplified}

In Section~\ref{Sect.Eval.AC}, we have found a DFAO that evaluates the abelian complexity function of the Tribonacci word. The automaton has $278$ states (including the initial state $q_0=0$). In this section we will show that the size of the automaton can be substantially reduced.

Obviously, if $p,q\in\{1,\ldots,277\}$ ($p\neq q$) satisfy
$$
\delta(p,0)=\delta(q,0) \quad\wedge\quad \delta(p,1)=\delta(q,1) \quad\wedge\quad \tau(p)=\tau(q)\,,
$$
then the states $p$ and $q$ can be merged into a single state.

We use this fact to reduce the size of our automaton; note that the procedure has to be repeated several times until no more mergeable vertices exist. Eventually we obtain an equivalent DFAO that has just $68$ vertices. Then we update the labels of states of the new automaton to the numbers $0,1,\ldots,67$ (where $0$ corresponds to the initial state) and recalculate accordingly the values of the transition function and the output function. The result is summarized in Table~\ref{Tab.red}.
\begin{table}
\begin{center}
\begin{tabular}{|c|cc|c||c|cc|c||c|cc|c|}
\hline
$q$ & $\delta(q,0)$ & $\delta(q,1)$ & $\tau(q)$ & $q$ & $\delta(q,0)$ & $\delta(q,1)$ & $\tau(q)$ & $q$ & $\delta(q,0)$ & $\delta(q,1)$ & $\tau(q)$ \\
\hline
 0  &  0  &  1  &  -  &  23  &  26  &  24  &  4  &  46  &  51  &  38  &  6  \\
 1  &  2  &  3  &  3  &  24  &  28  &  16  &  5  &  47  &  52  &  38  &  6  \\
 2  &  4  &  6  &  3  &  25  &  29  &  32  &  5  &  48  &  53  &  55  &  5  \\
 3  &  5  &  3  &  4  &  26  &  30  &  33  &  4  &  49  &  54  &  56  &  4  \\
 4  &  3  &  9  &  3  &  27  &  31  &  29  &  4  &  50  &  37  &  57  &  5  \\
 5  &  7  &  6  &  4  &  28  &  34  &  22  &  5  &  51  &  58  &  38  &  6  \\
 6  &  8  &  10  &  4  &  29  &  35  &  38  &  5  &  52  &  59  &  22  &  6  \\
 7  &  11  &  10  &  4  &  30  &  11  &  23  &  4  &  53  &  24  &  23  &  5  \\
 8  &  12  &  6  &  4  &  31  &  19  &  39  &  5  &  54  &  43  &  53  &  4  \\
 9  &  13  &  11  &  3  &  32  &  36  &  11  &  5  &  55  &  60  &  63  &  5  \\
 10  &  14  &  11  &  4  &  33  &  37  &  40  &  5  &  56  &  61  &  34  &  5  \\
 11  &  15  &  16  &  4  &  34  &  41  &  10  &  6  &  57  &  62  &  58  &  5  \\
 12  &  16  &  17  &  4  &  35  &  42  &  22  &  6  &  58  &  58  &  38  &  7  \\
 13  &  4  &  18  &  3  &  36  &  25  &  18  &  5  &  59  &  58  &  32  &  7  \\
 14  &  7  &  18  &  4  &  37  &  43  &  19  &  5  &  60  &  43  &  53  &  5  \\
 15  &  19  &  22  &  5  &  38  &  36  &  27  &  5  &  61  &  19  &  64  &  5  \\
 16  &  20  &  10  &  5  &  39  &  21  &  45  &  4  &  62  &  30  &  55  &  4  \\
 17  &  21  &  11  &  4  &  40  &  44  &  41  &  5  &  63  &  48  &  58  &  5  \\
 18  &  8  &  23  &  4  &  41  &  46  &  16  &  6  &  64  &  54  &  65  &  4  \\
 19  &  24  &  10  &  5  &  42  &  47  &  32  &  6  &  65  &  66  &  51  &  5  \\
 20  &  25  &  6  &  5  &  43  &  16  &  49  &  4  &  66  &  19  &  67  &  5  \\
 21  &  12  &  18  &  4  &  44  &  30  &  50  &  4  &  67  &  54  &  63  &  4  \\
 22  &  21  &  27  &  4  &  45  &  48  &  47  &  4  &  & & & \\
\hline
\end{tabular}
\end{center}
\caption{Reduced automaton. Values of the transition function $\delta(q,d)$ and of the output function $\tau$.}
\label{Tab.red}
\end{table}
Naturally, the formula
$$
\AC(n)=\tau(\delta(0,\langle n \rangle_T))
$$
from Theorem~\ref{MainThm} also remains valid with functions $\delta$ and $\tau$ tabulated in Table~\ref{Tab.red}.
Although the reduction method we used is quite elementary, it works well in this case: one can check, using, e.g., Moore's algorithm, that the reduced automaton is already minimal.

\begin{example}
Let us calculate $\AC(n)$ for $n=2013$ using Table~\ref{Tab.red}. Since $\langle2013\rangle_T=1001000101011$ (see Ex.~\ref{Ex.2013}), we obtain
$$
0\stackrel{1}{\rightarrow}1\stackrel{0}{\rightarrow}2\stackrel{0}{\rightarrow}4\stackrel{1}{\rightarrow}9\stackrel{0}{\rightarrow}13\stackrel{0}{\rightarrow}4\stackrel{0}{\rightarrow}3\stackrel{1}{\rightarrow}3\stackrel{0}{\rightarrow}5\stackrel{1}{\rightarrow}6\stackrel{0}{\rightarrow}8\stackrel{1}{\rightarrow}6\stackrel{1}{\rightarrow}10\,.
$$
Hence $\AC(2013)=\tau(10)=4$.
\end{example}

Now we can return to the open problem proposed by Richomme, Saari and Zamboni~\cite{RSZ2} concerning the characterization of those $n$ for which $\AC(n)=c$, $c\in\{4,5,6,7\}$.
The automaton derived in this section leads to a solution. Namely, for each $c\in\{4,5,6,7\}$ it suffices to consider the deterministic finite automaton (DFA) $A_c=(Q,\Sigma,\delta,q_0,F_c)$, where $Q,\Sigma,\delta,q_0$ have the same meaning as above, and $F_c\subset Q$ is the set of accepting states, defined by the equivalence $q\in F_c\Leftrightarrow\tau(q)=c$. The automaton $A_c$ obviously accepts $\langle n\rangle_T$ if and only if $\AC(n)=c$.

We noticed that Moore's minimization algorithm reduces the size of $A_c$ for each $c\in\{3,4,5,6,7\}$. It is interesting to compare the number of states of the reduced automata; see Table~\ref{Tab.Ac}.
\begin{table}
\begin{center}
\begin{tabular}{|c||c|c|c|c|c|}
\hline
$c$ & $3$ & $4$ & $5$ & $6$ & $7$ \\
\hline
number of states & $5$ & $66$ & $66$ & $66$ & $66$ \\
\hline
\end{tabular}
\end{center}
\caption{Sizes of minimal automata accepting the sets $\{\langle n\rangle_T\;|\;\AC(n)=c\}$.}
\label{Tab.Ac}
\end{table}
The minimal automaton accepting the set $\{\langle n\rangle_T\;|\;\AC(n)=3\}$ has just $5$ states (including the initial state), whereas the minimal automata accepting $\{\langle n\rangle_T\;|\;\AC(n)=c\}$ for $c\in\{4,5,6,7\}$ have $66$ states. The disparity in sizes explain, in our opinion, why the authors of~\cite{RSZ2} could give a simple characterization of those $n$ satisfying $AC(n)=3$, whereas the question about the $n$ satisfying $AC(n)=c$ for $c\in\{4,5,6,7\}$ remained open.

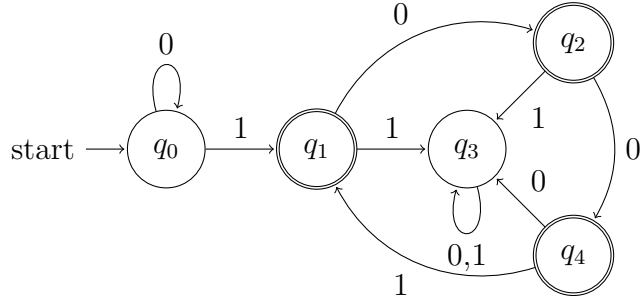
\begin{figure}
\begin{center}
\begin{tikzpicture}[shorten >=1pt,node distance=2cm,on grid,auto] 
   \node[state,initial] (q_0)   {$q_0$}; 
   \node[state,accepting] (q_1) [right=of q_0] {$q_1$}; 
   \node[state] (q_3) [right=of q_1] {$q_3$}; 
   \node[state,accepting] (q_2) [above right=of q_3] {$q_2$}; 
   \node[state,accepting] (q_4) [below right=of q_3] {$q_4$}; 
    \path[->] 
    (q_0) edge node {1} (q_1)
          edge [loop above] node {0} ()
    (q_1) edge [bend left=40] node {0} (q_2)
          edge node {1} (q_3)
    (q_2) edge [bend left] node {0} (q_4)
          edge node {1} (q_3)
    (q_3) edge [loop below] node {0,1} ()
    (q_4) edge node [swap] {0} (q_3)
          edge [bend left=40] node {1} (q_1);
\end{tikzpicture}
\end{center}
\caption{The minimal automaton accepting the set $\{\langle n\rangle_T\;|\;\AC(n)=3\}$.}
\label{Autom.3}
\end{figure}

\begin{remark}\label{Rem.AC3}
The minimal automaton corresponding to $A_3$, which accepts $\langle n\rangle_T$ if and only if $\AC(n)=3$, has the transition diagram shown in Figure~\ref{Autom.3}.
Hence we obtain one more characterization of those $n$ for which $\AC(n)=3$ on top of those that were found in~\cite[Prop.~3.3]{RSZ2}, namely:
\begin{equation}\label{AC3}
\AC(n)=3 \qquad\Leftrightarrow\qquad \langle n\rangle_T \text{ is a prefix of } 100100100\cdots.
\end{equation}
\end{remark}

\section{On the abelian complexity of $m$-bonacci words for $m \geq 4$}\label{Sect.mbon}

The $m$-bonacci word is defined for any integer $m\geq2$ over the alphabet $\A=\{0,1,\ldots,m-1\}$ as the fixed point of the substitution
$$
\varphi_m:\qquad 0\mapsto01,\;1\mapsto02,\;\ldots,\;m-2\mapsto0(m-1),\;m-1\mapsto0\,.
$$
The Fibonacci word and the Tribonacci word are its special cases for $m=2$ and $m=3$, respectively.

It is easy to see that the procedure, used in previous sections to examine the abelian complexity of the Tribonacci word, can be straightforwardly applied to any $m$-bonacci word, regardless of $m$.
Indeed, to explore an $m$-bonacci word for a given $m\geq2$ by this method, it suffices to change just the constant $R$ in Example~\ref{Z(1)} from the value $3$ to $m$, and to use the $m$-bonacci representation of integers. The $m$-bonacci representation is a normal $U$-representation defined for $U_j=|\varphi_m^j(0)|$; note that the values $U_j$ satisfy
$$
U_j=\begin{cases}
2^j, & \text{if } j\in\{0,1,\ldots,m-1\}; \\
U_{j-1}+U_{j-2}+\cdots U_{j-m}, & \text{if } j\geq m.
\end{cases}
$$
On the other hand, the cardinality of $\Z_\mathrm{super}$ seems to quickly grow with $m$; thus the method ceases to be efficient for high values of $m$.

For instance, we have considered the $4$-bonacci word \seqnum{A254990} and successfully found an explicit DFAO that evaluates its abelian complexity \seqnum{A255014}; the automaton has $5665$ states ($66881$ before the reduction of states). We provide main results below. For the sake of brevity we will denote the $4$-bonacci word by the symbol $\q$. For any $n\in\N$, let $\langle n\rangle_Q$ be the $4$-bonacci representation of $n$, constructed for $U_{j}=Q_{j+4}$, where $\left(Q_j\right)_{j\geq0}=(0,0,0,1,1,2,4,8,15,\ldots)$ is the sequence of Tetranacci numbers \seqnum{A000078}.

The minimal and maximal values of the output function of the automaton evaluating $\AC_\q$ are $4$ and $16$, respectively. Therefore, the abelian complexity function $\AC_\q$ takes values between $4$ and $16$. However, the output function of the automaton does not attain the value $5$, which implies that there exists no $n\in\N$ such that $\AC_\q(n)=5$. To sum up,
$$
\left\{\AC_\q(n)\;|\; n\in\N\right\}=\{4\}\cup\{6,7,\ldots,16\}.
$$
The existence of gaps in ranges of abelian complexity functions of $m$-bonacci words was already observed a few years ago by K.~B\v{r}inda~\cite{Brinda} on the basis of computer-assisted calculations performed for $m\in\{4,5,\ldots,12\}$. Our automaton confirms his observation in the case $m=4$.

Furthermore, we are able to characterize 
those $n$ for which $\AC_\q(n)=4$. We apply Moore's minimization algorithm on the automaton accepting the set $\{\langle n\rangle_T\;|\;\AC_\q(n)=4\}$, which leads to an automaton having just $6$ states. The transition diagram of the automaton is depicted in Figure~\ref{Autom4.4}.
\begin{figure}
\begin{center}
\begin{tikzpicture}[shorten >=1pt,node distance=2cm,on grid,auto] 
   \node[state,initial] (q_0)   {$q_0$}; 
   \node[state,accepting] (q_1) [right=of q_0] {$q_1$}; 
   \node[state] (q_3) [right=of q_1] {$q_3$}; 
   \node[state,accepting] (q_2) [above=of q_3] {$q_2$}; 
   \node[state,accepting] (q_4) [right=of q_3] {$q_4$}; 
   \node[state,accepting] (q_5) [below=of q_3] {$q_5$}; 
    \path[->] 
    (q_0) edge node {1} (q_1)
          edge [loop above] node {0} ()
    (q_1) edge [bend left] node {0} (q_2)
          edge node {1} (q_3)
    (q_2) edge [bend left] node {0} (q_4)
          edge node {1} (q_3)
    (q_3) edge [out=240,in=210,loop] node {0,1} ()
    (q_4) edge [bend left] node {0} (q_5)
          edge node [swap] {1} (q_3)
    (q_5) edge node [swap] {0} (q_3)
          edge [bend left] node {1} (q_1);
\end{tikzpicture}
\end{center}
\caption{($4$-bonacci word) The minimal automaton accepting the set $\left\{\langle n\rangle_Q\;|\;\AC_\q(n)=4\right\}$.}
\label{Autom4.4}
\end{figure}
One can see directly from its structure that
\begin{equation}\label{AC4}
\AC_\q(n)=4 \qquad\Leftrightarrow\qquad \langle n\rangle_Q \text{ is a prefix of } 100010001000\cdots.
\end{equation}
This result is the $4$-bonacci version of the Tribonacci equivalence~\eqref{AC3}.

Let us proceed to characterization of those $n$ such that  $\AC_\q(n)=6$. We again apply Moore's minimization, and again get an automaton having only $6$ states; see Figure~\ref{Autom4.6}.
\begin{figure}
\begin{center}
\begin{tikzpicture}[shorten >=1pt,node distance=2cm,on grid,auto] 
   \node[state,initial] (q_0)   {$q_0$}; 
   \node[state] (q_1) [right=of q_0] {$q_1$}; 
   \node[state] (q_2) [above right=of q_1] {$q_2$}; 
   \node[state,accepting] (q_3) [below right=of q_1] {$q_3$}; 
   \node[state,accepting] (q_4) [right=of q_3] {$q_4$}; 
   \node[state,accepting] (q_5) [right=of q_2] {$q_{5}$}; 
    \path[->] 
    (q_0) edge node {1} (q_1)
          edge [loop above] node {0} ()
    (q_1) edge node {0} (q_2)
          edge node [swap] {1} (q_3)
    (q_2) edge [loop above] node {0,1} ()
    (q_3) edge node {0} (q_4)
          edge node {1} (q_2)
    (q_4) edge node [swap] {0} (q_5)
          edge node [swap] {1} (q_2)
    (q_5) edge node [swap] {0,1} (q_2);
\end{tikzpicture}
\end{center}
\caption{($4$-bonacci word) The minimal automaton accepting the set $\left\{\langle n\rangle_Q\;|\;\AC_\q(n)=6\right\}$.}
\label{Autom4.6}
\end{figure}
It is obvious from the graph that
$$
\AC_\q(n)=6 \qquad\Leftrightarrow\qquad \langle n\rangle_Q\in\{11,110,1100\},
$$
i.e., the value $6$ is attained only for $n\in\{3,6,12\}$. The fact that the function $\AC_\q$ takes a certain value only finitely many times is remarkable, because it implies a significant qualitative difference between abelian complexity functions of the Tribonacci and the $4$-bonacci word. Recall that each value in the range of $\AC_\t$ is attained infinitely often~\cite{RSZ2,Tu13}.

Now we will comment on the remaining values of $\AC_\q$, i.e., $c\in\{7,\ldots,16\}$. Although one can again construct minimal automata accepting $\{\langle n\rangle_Q\;|\;\AC_\q(n)=c\}$, they are not useful for obtaining a nice characterization of the numbers $n$ such that $\AC_\q(n)=c$. This follows from Table~\ref{Tab.Ac4}:
\begin{table}
\begin{center}
\begin{tabular}{|c||c|c|c|c|c|c|c|c|c|c|c|c|c|}
\hline
$c$ & $4$ & $5$ & $6$ & $7$ & $8$ & $9$ & $10$ & $11$ & $12$ & $13$ & $14$ & $15$ & $16$ \\
\hline
states & $6$ & -- & $6$ & $66$ & $4649$ & $4683$ & $4735$ & $5004$ & $5256$ & $5299$ & $5322$ & $5324$ & $5032$ \\
\hline
\end{tabular}
\end{center}
\caption{($4$-bonacci word) Sizes of minimal automata accepting the sets $\{\langle n\rangle_Q\;|\;\AC_\q(n)=c\}$.}
\label{Tab.Ac4}
\end{table}
the automata are too large, especially for $c\geq8$. Nevertheless, we are still able to demonstrate that each value in the set $\{7,\ldots,16\}$ is attained infinitely many times. For every $c\in\{7,\ldots,16\}$, we give in Table~\ref{Tab.inf} an infinite family of normal $Q$-representations $\langle n\rangle_Q$ such that $\AC_\q(n)=c$.
\begin{table}
\begin{center}
\begin{tabular}{|c|cc||c|cc|}
\hline
$c$ & $\langle n\rangle_Q$ &  & $c$ & $\langle n\rangle_Q$ & \\
\hline
$7$ & $(1000)^j0$ & ($\forall j\geq1$) & $12$ & $10^j1$ & ($\forall j\geq19$) \\
$8$ & $(100)^j$ & ($\forall j\geq3$) & $13$ & $10^j$ & ($\forall j\geq19$) \\
$9$ & $(10)^j$ & ($\forall j\geq11$) & $14$ & $(10000)^j$ & ($\forall j\geq6$) \\
$10$ & $10^j11$ & ($\forall j\geq19$) & $15$ & $(10000)^j0$ & ($\forall j\geq6$) \\
$11$ & $(10)^j0$ & ($\forall j\geq11$) & $16$ & $(10000)^j00$ & ($\forall j\geq6$) \\
\hline
\end{tabular}
\end{center}
\caption{($4$-bonacci word) Examples of infinite families of normal $Q$-representations $\langle n\rangle_Q$ such that $\AC_\q(n)=c$ for $c\in\{7,\ldots,16\}$. The data in the table were obtained by trial and error: we used the automaton evaluating $\AC_\q(n)$ to explore several periodic expansions, some with an aperiodic part at the end, and noted down those expansions that were useful. Many other such examples can be found.}
\label{Tab.inf}
\end{table}

Let us summarize the results obtained on the $4$-bonacci word.

\begin{theorem}
The abelian complexity function of the $4$-bonacci word has the following properties.
\begin{itemize}
\item $\left\{\AC_\q(n)\;|\; n\in\N\right\}=\{4\}\cup\{6,7,\ldots,16\}$.
\item $\AC_\q(n)=4$ if and only if $\langle n\rangle_Q$ is a prefix of $100010001000\cdots$.
\item $\AC_\q(n)=6$ if and only if $n\in\{3,6,12\}$.
\item For every value $c\in\{7,\ldots,16\}$ there exist infinitely many integers $n\in\N$ such that $\AC_\q(n)=c$.
\end{itemize}
\end{theorem}

We are convinced that the existence of gaps in the range of the abelian complexity function, as well as the existence of values that are attained only finitely many times, are common properties of all $m$-bonacci words with $m\geq4$.

We finish the section by a remark on the minimal value of the abelian complexity function of the $m$-bonacci word for a general $m$.
Let $\u$ be an $m$-bonacci word. One can easily show that for every $n\in\N$ and $\ell\in\{0,1,\ldots,m-1\}$, the factor $\ell\u_{[n-1]}$ is a factor of $\u$, thus $\AC_\u(n)\geq m$. At the same time we have $\AC_\u(1)=m$. To sum up, $\min_{n\in\N}\AC_\u(n)=m$ for any $m$-bonacci word.
Results of K.~B\v{r}inda's calculations~\cite{Brinda} together with proven formulas~\eqref{AC3} and \eqref{AC4} suggest a conjecture on a precise characterization of the numbers for which the abelian complexity function of an $m$-bonacci word $\u$ attains its minimum:
\begin{equation}\label{ACm}
\AC_\u(n)=m \qquad\Leftrightarrow\qquad \langle n\rangle_U \text{ is a prefix of } 10^{m-1}10^{m-1}10^{m-1}\cdots;
\end{equation}
the symbol $\langle n\rangle_U$ stands here for the $m$-bonacci representation of $n$. We are able to prove the implication $\Leftarrow$ in~\eqref{ACm} for a general $m$ by the abelian co-decomposition method, introduced earlier~\cite{Tu13}. The implication $\Rightarrow$ remains so far open, although we expect that it is probably not difficult to be proven either.

\section{Conclusions and generalizations}\label{Sect.Conclusions}

In this paper we focused on the abelian complexity of the Tribonacci word (or, more generally, $m$-bonacci words), but the method can be easily adapted for application on any simple Parry word. Let us consider the fixed point $\u$ of a substitution~\eqref{simpleParry}. The calculation naturally takes advantage of the numeration system associated with $\varphi$, i.e., of the normal $U$-representation for $U_j=|\varphi^j(0)|$. Let us briefly sketch the procedure. First of all, for the sake of generality, we slightly reformulate the definition of $\Codec\bigl(\begin{smallmatrix}v\\ w\end{smallmatrix}\bigr)$ by imposing an additional technical assumption on the decomposition~\eqref{rozklad}. Namely, we assume that for every $j\in\{1,\ldots,h\}$, the factor $\tilde{z}$ has the prefix $0$, and require that $h$ is maximal subject to this condition. We also need to introduce maps $\D_0,\D_1,\ldots,\D_{\alpha_0}$, defined in a way analogous to equations~\eqref{D0D1},
i.e.,
$$
\D_j(\zeta):=\Codec\begin{pmatrix}\varphi(z)\\ 0^{-j}\varphi(\tilde{z})0^j\end{pmatrix} \qquad\text{for }\zeta=\binom{z}{\tilde{z}}\,,\; j\in\{0,1,\ldots,\alpha_0\}\,.
$$
The search for sets $\Z_1,\ldots,\Z_M$ starts with calculating $\Z_\u(n)$ by formula~\eqref{Z(n)} for every $n\in\{1,\ldots,\alpha_0\}$. Then we take the bunch of sets $\Z_\u(1),\ldots,\Z_\u(\alpha_0)$, which we denote by $\Z_1,\ldots,\Z_{\alpha_0}$, and apply the maps $\D_0,\D_1,\ldots,\D_{\alpha_0}$ on each of the sets, similarly as in the proof of Theorem~\ref{Subsets}. In this way we obtain a new bunch of sets, we apply $\D_0,\D_1,\ldots,\D_{\alpha_0}$ on them again, and so forth. However, unlike in the case of $m$-bonacci words, one needs to keep track of the admissibility of the normal $U$-representations $\langle n\rangle_U$ during the calculation. Briefly speaking, if a normal $U$-representation, examined at a given moment, cannot be validly extended by a certain specific digit $d$, then the map $\D_d$ is not applied at that stage. The procedure ends when the application of $\D_0,\D_1,\ldots,\D_{\alpha_0}$ generates no new data.

The abelian co-decomposition method also allows us to deal with the other type of Parry words, called non-simple Parry words, which are fixed points of substitutions
$$
\begin{array}{lccl}
\varphi: & 0 & \mapsto & 0^{\alpha_0}1 \\
& 1 & \mapsto & 0^{\alpha_1}2 \\
&  & \vdots & \\
& m & \mapsto & 0^{\alpha_{m}}(m+1) \\
&  & \vdots & \\
& m+p-2 & \mapsto & 0^{\alpha_{m+p-2}}(m+p-1) \\
& m+p-1 & \mapsto & 0^{\alpha_{m+p-1}}m
\end{array}
$$
where $\alpha_j$ satisfy $\alpha_0\geq1$, $\alpha_\ell\leq\alpha_0$ for all $\ell\in\A$, and $(\exists\ell\in\{m,m+1,\ldots,m+p-1\})(\alpha_\ell\geq1)$.
Although we have not explicitly discussed non-simple Parry words in previous sections, the implementation of the procedure would be the same; we just need to replace the value $m$ in the definition of $R$~\eqref{R} by $m+p$. To sum up, the approach is applicable on any Parry word; but one shall keep in mind that in practice it will more likely work well in cases when the image of the abelian complexity function is a set of low cardinality. Nevertheless, it can still give new results for various words for which other methods fail.

Those Parry words, for which this approach turns out to be inefficient, can be perhaps treated by a newer technique, which replaces pairs $\binom{z}{\tilde{z}}$ by certain conveniently chosen triples~\cite{Tu15}. That technique is more involved, but it is expected to have smaller memory requirements in most cases and to work faster.

A potentially interesting question is whether this approach (possibly after a certain improvement) can be used for dealing with a word that depends on a parameter, i.e., whether one can explore a parametric family of words as a whole. Consider for instance the $m$-bonacci word for a general $m\geq2$. We are convinced that the procedure could be implemented with a parameter as well, although the calculation would be of course intricate and lengthy.

The procedure also gives, as a by-product, the optimal balance bound of the examined word. The optimal bound is equal to the maximal entries of vectors $\psi$ having the form $\psi=v_i-v_j$ for $v_i,v_j$ belonging to the same set $\bigcup_{\zeta\in\Z_q}\Vect(\zeta)$. For example, one can check in this way that the $4$-bonacci word is $3$-balanced. Consequently, we can regard the method not only as a tool for evaluating abelian complexity, but also as a tool for exploring balance properties of words. In particular, it is possible that this approach can lead to the optimal balance bound for the $m$-bonacci word for any $m$. Recall that the optimal bound for the $m$-bonacci word is not known yet, despite the fact that an upper bound has been already determined~\cite{BPT13}.

\section{Acknowledgements}

The author thanks J.-P. Allouche for useful comments and suggestions, and the referees for valuable hints and corrections that helped to improve the manuscript.

\bigskip
\hrule
\bigskip

\noindent 2010 {\it Mathematics Subject Classification}:
Primary 11B85; Secondary 68R15.

\noindent \emph{Keywords: } abelian complexity, Tribonacci word, finite automaton, $4$-bonacci word.

\bigskip
\hrule
\bigskip

\noindent (Concerned with sequences
\seqnum{A000073},
\seqnum{A000078},
\seqnum{A080843},
\seqnum{A216190},
\seqnum{A254990}, and
\seqnum{A255014}.)

\bigskip
\hrule
\bigskip

\vspace*{+.1in}
\noindent
Received October 7 2014;
revised version received February 12 2015.
Published in {\it Journal of Integer Sequences},
February 14 2015.

\bigskip
\hrule
\bigskip

\noindent
Return to
\htmladdnormallink{Journal of Integer Sequences home page}{http://www.cs.uwaterloo.ca/journals/JIS/}.
\vskip .1in

\end{document}